\documentclass[a4paper,leqno,12pt]{article}

\usepackage[utf8]{inputenc}
\usepackage[T1]{fontenc}
\usepackage[english]{babel}

\usepackage{crimson}

\usepackage{geometry}

\usepackage{fancyhdr}

\usepackage{hyperref}
\hypersetup{
  colorlinks   = true, 
  urlcolor     = blue, 
  linkcolor    = blue, 
  citecolor   = black 
}

\usepackage[inline]{enumitem}

\usepackage{graphicx}

\usepackage{color}

\usepackage[intlimits]{amsmath}
\usepackage{amssymb, amsfonts}

\usepackage[thmmarks, amsmath, amsthm]{ntheorem}

\usepackage{MnSymbol}
\usepackage{mathbbol}

\usepackage{bm}

\usepackage{bbm}

\usepackage[all]{xy}

\usepackage{tikz-cd}

\usepackage[nameinlink, capitalize, noabbrev]{cleveref}

\numberwithin{equation}{section}
\newtheorem{theoremcounter}{theoremcounter}[section]

\theoremstyle{plain}

\newtheorem{corollary}[theoremcounter]{Corollary}
\newtheorem{lemma}[theoremcounter]{Lemma}
\newtheorem{proposition}[theoremcounter]{Proposition}

\newtheorem{theorem}[theoremcounter]{Theorem}

\theoremnumbering{Alph}
\newtheorem{introtheorem}{Theorem}

\theoremstyle{definition}

\newtheorem{convention}[theoremcounter]{Convention}
\newtheorem{definition}[theoremcounter]{Definition}

\theoremstyle{remark}

\newtheorem{notation}[theoremcounter]{Notation}

\newtheorem{remark}[theoremcounter]{Remark}

\usepackage{xargs}                      
\usepackage[colorinlistoftodos,prependcaption,textsize=tiny]{todonotes}
\newcommandx{\unsure}[2][1=]{\todo[linecolor=red,backgroundcolor=red!25,bordercolor=red,#1]{#2}}
\newcommandx{\change}[2][1=]{\todo[linecolor=blue,backgroundcolor=blue!25,bordercolor=blue,#1]{#2}}
\newcommandx{\info}[2][1=]{\todo[linecolor=OliveGreen,backgroundcolor=OliveGreen!25,bordercolor=OliveGreen,#1]{#2}}
\newcommandx{\improvement}[2][1=]{\todo[linecolor=Plum,backgroundcolor=Plum!25,bordercolor=Plum,#1]{#2}}

\usepackage[style=alphabetic, backend=bibtex]{biblatex} 
\DeclareFieldFormat{eprint}{\href{https://arxiv.org/abs/#1}{\texttt{#1}}}

\renewbibmacro*{in:}{}

\DeclareFieldFormat
[article,inbook,incollection,inproceedings,patent,thesis,unpublished,misc]
{title}{#1}

\newcommand{\cA}{\ensuremath{\mathcal{A}}}
\newcommand{\cB}{\ensuremath{\mathcal{B}}}
\newcommand{\cC}{\ensuremath{\mathcal{C}}}

\newcommand{\cE}{\ensuremath{\mathcal{E}}}
\newcommand{\cF}{\ensuremath{\mathcal{F}}}
\newcommand{\cG}{\ensuremath{\mathcal{G}}}
\newcommand{\cH}{\ensuremath{\mathcal{H}}}

\newcommand{\cN}{\ensuremath{\mathcal{N}}}

\newcommand{\cS}{\ensuremath{\mathcal{S}}}

\newcommand{\cU}{\ensuremath{\mathcal{U}}}

\newcommand{\cZ}{\ensuremath{\mathcal{Z}}}

\newcommand{\bT}{\ensuremath{\mathbb{T}}}

\newcommand{\rE}{\ensuremath{\mathrm{E}}}
\newcommand{\rF}{\ensuremath{\mathrm{F}}}

\newcommand{\rH}{\ensuremath{\mathrm{H}}}

\newcommand{\rL}{\ensuremath{\mathrm{L}}}
\newcommand{\rM}{\ensuremath{\mathrm{M}}}

\newcommand{\rR}{\ensuremath{\mathrm{R}}}
\newcommand{\rS}{\ensuremath{\mathrm{S}}}

\newcommand{\rZ}{\ensuremath{\mathrm{Z}}}

\newcommand{\rmd}{\ensuremath{\mathrm{d}}}

\newcommand{\rmh}{\ensuremath{\mathrm{h}}}

\newcommand{\rms}{\ensuremath{\mathrm{s}}}

\newcommand{\veps}{\ensuremath{\varepsilon}}

\newcommand{\vphi}{\ensuremath{\varphi}}

\newcommand{\ol}{\overline}

\newcommand{\eqstop}{\ensuremath{\, \text{.}}}
\newcommand{\eqcomma}{\ensuremath{\, \text{,}}}

\newcommand{\NN}{\ensuremath{\mathbb{N}}}
\newcommand{\ZZ}{\ensuremath{\mathbb{Z}}}

\newcommand{\RR}{\ensuremath{\mathbb{R}}}
\newcommand{\CC}{\ensuremath{\mathbb{C}}}

\newcommand{\id}{\ensuremath{\mathrm{id}}}

\newcommand{\lra}{\ensuremath{\longrightarrow}}
\newcommand{\hra}{\ensuremath{\hookrightarrow}}
\newcommand{\thra}{\ensuremath{\twoheadrightarrow}}

\newcommand{\GL}{\operatorname{GL}}

\newcommand{\Aut}{\ensuremath{\mathrm{Aut}}}

\newcommand{\ot}{\ensuremath{\otimes}}

\newcommand{\Cstar}{\ensuremath{\mathrm{C}^*}}

\newcommand{\bo}{\ensuremath{\mathcal{B}}}

\newcommand{\supp}{\ensuremath{\mathop{\mathrm{supp}}}}

\newcommand{\Cstarred}{\ensuremath{\Cstar_\mathrm{red}}}

\newcommand{\cont}{\ensuremath{\mathrm{C}}}

\newcommand{\conto}{\ensuremath{\mathrm{C}_0}}
\newcommand{\contc}{\ensuremath{\mathrm{C}_\mathrm{c}}}

\newcommand{\Lone}{\ensuremath{{\offinterlineskip \mathrm{L} \hskip -0.3ex ^1}}}
\newcommand{\ltwo}{\ensuremath{\ell^2}}

\newcommand{\Ad}{\ensuremath{\mathop{\mathrm{Ad}}}}

\newcommand{\grpaction}[1]{\ensuremath{\stackrel{#1}{\curvearrowright}}}

\newcommand{\lone}{\ensuremath{\ell^1}}
\newcommand{\redtimes}{\ensuremath{\rtimes_{\mathrm{red}}}}
\newcommand{\res}{\ensuremath{\mathrm{res}}}

\newcommand{\unitspace}[1]{#1^{(0)}}
\newcommand{\IsoInt}[1]{\mathcal{I}^{#1}}

\newcommand{\authors}{Are Austad and Sven Raum}
\renewcommand{\title}{Detecting ideals in reduced crossed product C*-algebras of topological dynamical systems}

\begin{document}


\thispagestyle{empty}

\noindent
\begin{minipage}{\linewidth}
  \begin{center}
    \textbf{\Large \title} \\[0.2em]
    \authors    
  \end{center}
\end{minipage}

\renewcommand{\thefootnote}{}
 \footnotetext{last modified on \today}
\footnotetext{
  \textit{MSC 2020 classification: 46L05, 37B02, 22E40, 20G30, 20F16}}

\footnotetext{
  \textit{Keywords: ideal intersection property, crossed product \Cstar-algebra, topological dynamical system, lattices in Lie groups, linear groups, polycyclic groups}
}

\vspace{2em}
\noindent
\begin{minipage}{\linewidth}
  \textbf{Abstract}.
  We introduce the \lone-ideal intersection property for crossed product \Cstar-algebras. It is implied by \Cstar-simplicity as well as \Cstar-uniqueness.  We show that topological dynamical systems of arbitrary lattices in connected Lie groups, arbitrary linear groups over the integers in a number field and arbitrary virtually polycyclic groups have the \lone-ideal intersection property.  On the way, we extend previous results on \Cstar-uniqueness of \Lone-groupoid algebras to the general twisted setting.
\end{minipage}


\section{Introduction}
\label{sec:introduction}

Crossed products associated with topological dynamical systems are among the prime sources of examples in the theory of \Cstar-algebras.  In recent years, amenable dynamical systems received abundant attention in the context of Elliott's classification programme (see e.g. \cite{hirshbergwinterzacharias2015,szabo2015-zn-actions,deeleyputnamstrung2015,kerrszabo2020}), while reduced group \Cstar-algebras were put in the spotlight by breakthrough results on \mbox{\Cstar-simplicity} \cite{kalantarkennedy14-boudaries,breuillardkalantarkennedyozawa14,kennedy15-cstarsimplicity,haagerup15}.

A foundational problem about crossed product \Cstar-algebras concerns their ideal structure.  For tame dynamical systems it is possible to give a complete description of the primitive ideal space of the associated crossed product in terms of induced primitive ideals, thanks to the Mackey machine \cite{rosenberg94,echterhoffwilliams2008}.  For wilder dynamical systems and for group \Cstar-algebras, it is the question of simplicity that received most attention.  Following seminal work on simplicity of group \Cstar-algebras \cite{kalantarkennedy14-boudaries,breuillardkalantarkennedyozawa14}, a satisfactory characterisation of topological dynamical systems whose crossed product \Cstar-algebras are simple could be obtained in \cite{kawabe17}.  This line of research even led to complete results about simplicity of \Cstar-algebras associated with {\'e}tale groupoids \cite{borys2019-boundary,kennedykimliraumursu2021}.  One important insight from the study of the ideal structure of groupoid \Cstar-algebras originated in \cite{tomiyama1992-lecture-notes} and says that specific subalgebras have the potential to detect ideals.

Much fewer results are available for dynamical systems that neither are tame nor give rise to simple crossed products.  However, the idea of employing subalgebras to detect ideals had surfaced earlier in a completely different context.  In work on abstract harmonic analysis and representation theory of solvable Lie groups, the concept of \Cstar-uniqueness of \Lone-convolution algebras was introduced in the late 1970's and early 1980's \cite{boidolleptinschurmannwahle1978,boidol1984}.  This notion can be reformulated as an ideal intersection property for the inclusion $\Lone(G) \subseteq \Cstar(G)$.  While for exponential solvable Lie groups Boidol could establish conclusive results \cite{boidol1980}, there have been no noteworthy advances in the investigation of \Cstar-uniqueness of discrete groups beyond the positive results for virtually nilpotent groups and some metabelian groups in the late 1970's \cite{boidolleptinschurmannwahle1978}.  It is considered an open question whether every amenable discrete group is \Cstar-unique \cite[Remark 3.6]{leungng2004}.  In some recent work starting with \cite{grigorchukmusatrordam17}, variations of \Cstar-uniqueness replacing the \lone-convolution algebra of a discrete group by its complex group algebra have been considered.  Results from \cite{alekseevkyed2019,alekseev2019-MFO-report,scarparo2020-cstar-unique} create an unclear image of which groups might have this property termed algebraic \mbox{\Cstar-uniqueness}.

The aim of this article is to introduce and study a new ideal intersection property for crossed product \Cstar-algebras, which can be established for a large class of examples including all \Cstar-simple groups and all \Cstar-unique groups.  At present, we have no example of a topological dynamical system $\Gamma \grpaction{} X$ that fails the following \emph{\lone-ideal intersection property} (compare with Definition~\ref{def:l1-ideal-intersection-property-crossed-products}): every non-zero ideal of the \Cstar-algebraic crossed product $\conto(X) \rtimes_{\mathrm{red}} \Gamma$ has non-zero intersection with the \lone-crossed product $\conto(X) \rtimes_{\lone} \Gamma$.  The next two theorems describe large classes of non-amenable and amenable examples of groups for which we can establish the \lone-ideal intersection property.
\begin{introtheorem}[See \cref{cor:acylindrically hyperbolic-groups-lone-ideal-intersection-property}, \cref{cor:lattices-lone-ideal-intersection-property} and \cref{cor:linear-groups-lone-ideal-intersection-property}]
  \label{introthm:non-amenable}
  Let $\Gamma$ be either an acylindrically hyperbolic group, a lattice in a connected Lie group or a linear group over the integers in a number field.  Then every action of $\Gamma$ on a locally compact Hausdorff space has the \lone-ideal intersection property.
\end{introtheorem}
We remark that (acylindrically) hyperbolic groups, mapping class groups, and arithmetic groups fall in the scope of our theorem.
For amenable groups the \lone-ideal intersection property coincides with the notion of \Cstar-uniqueness, so that we obtain the first major enlargement since \cite{boidolleptinschurmannwahle1978} of the class of groups for which this property is known.
\begin{introtheorem}[See \cref{cor:locally-virtually-polycyclic}]
  \label{introthm:amenable}
  Every action of a locally virtually polycyclic group on a locally compact Hausdorff space has the \lone-ideal intersection property.  In particular, every locally virtually polycyclic group is \Cstar-unique.
\end{introtheorem}
We remark that also metabelian groups can be covered by our methods, as noted in \cref{rem:possible-extensions}.  Thus our results cover and extend all previously known examples of \Cstar-unique groups.
In order to obtain the results above, we establish a general criterion on groups to satisfy the \lone-ideal intersection property for all their dynamical systems.  It combines assumptions that relate to \mbox{\Cstar-simplicity} with conditions on the structure of amenable subgroups.
\begin{introtheorem}[See \cref{thm:ideal-intersection-general-form}]
  \label{introthm:general}
  Let $\Gamma$ be a discrete group such that the following three conditions hold for every finitely generated subgroup $\Lambda \leq \Gamma$:
  \begin{itemize}
  \item the amenable radical of $\Lambda$ is a Furstenberg subgroup,
  \item every amenable subgroup of $\Lambda$ is virtually solvable, and
  \item there is $l \in \NN$ such that every solvable subgroup of $\Lambda$ is polycyclic of Hirsch length at most $l$.
  \end{itemize}
  Then every action of $\Gamma$ on a locally compact Hausdorff space has the \lone-ideal intersection property.
\end{introtheorem}
The proof of \cref{introthm:general} leverages and extends recent developments in the structure theory of groupoid operator algebras in order to set up an induction scheme.  The assumptions on amenable subgroups are exploited to construct a twisted groupoid, which is analysed from the point of view of the \Lone-ideal intersection property for groupoids, previously studied in \cite{austadortega2022}.  Ultimately our induction argument reduces the maximal possible Hirsch length of polycyclic subgroups.  The following generalisation of work from \cite{austadortega2022} allows us to analyse the twisted groupoids we obtain when applying our strategy.  It might be of independent interest.
\begin{introtheorem}[See \cref{cor:groupoid-lone-ideal-intersection-property}]
  \label{introthm:groupoids}
  Let $\cE$ be a twist over a second-countable locally compact {\'e}tale Hausdorff groupoid $\cG$.  Assume that there is a dense subset $D \subseteq \unitspace{\cG}$ such the fibres $(\IsoInt{\cG}_x, \IsoInt{\cE}_x)$ have the \lone-ideal intersection property for all $x \in D$.  Then $(\cG, \cE)$ has the \Lone-ideal intersection property.
\end{introtheorem}

\subsection*{Structure of the article}

In \cref{sec:preliminaries}, we introduce necessary background material and fix notation.  In \cref{sec:lone-ideal-intersection-property-basics} we formally define the \lone-ideal intersection property for twisted \Cstar-dynamical systems and show in particular that it is closed under directed unions of groups.  In \cref{section:groupoid-result-extension} we study the \Lone-ideal intersection property for twisted groupoids.  In \cref{sec:groupoid-presentation} we present certain twisted group \Cstar-algebras as twisted groupoid \mbox{\Cstar-algebras}, which is a key ingredient for the subsequent \cref{sec:main-results}, where we obtain our main results.

\subsection*{Acknowledgements}

The first author gratefully acknowledges the financial support from the Independent Research Fund Denmark through grant number 1026-00371B and The Research Council of Norway project 324944. The second author was supported by the Swedish Research Council through grant number 2018-04243. The authors would like to thank the organisers of the 28th Nordic Congress of Mathematicians in Aalto, where this project was initiated. They are grateful to Matthew Kennedy for interesting discussions about the \lone-ideal intersection property and to Magnus Goffeng for asking whether beyond group algebras also crossed product \Cstar-algebras could be investigated with the present techniques. We thank Becky Armstrong for clarifying conversations on the role of the second-countability assumption in her work.

\section{Preliminaries}
\label{sec:preliminaries}

\begin{convention}
  \label{conv:discrete-groups}
  All groups in this article are discrete unless otherwise specified.  Every ideal in a \Cstar-algebra will be assumed to be closed and two-sided. 

\end{convention}

\subsection{Virtually polycyclic groups}
\label{sec:virtually-polycyclic-groups}

A group $\Gamma$ is called \emph{polycyclic} if there exists a subnormal series
\begin{equation*}
  1 = \Gamma_0 \unlhd \Gamma_1 \cdots \unlhd \Gamma_{n-1} \unlhd \Gamma_n = \Gamma
\end{equation*}
such that each of the factor groups $\Gamma_i /\Gamma_{i-1}$ is cyclic. The group is called \emph{poly-$\ZZ$} if each of the factor groups are isomorphic to $\ZZ$.

It is known that polycyclic groups are precisely the solvable groups for which every subgroup is finitely generated, see \cite[Chapter 1, Proposition 4]{segal1983}.  We also recall from \cite[Chapter 1, Proposition 2]{segal1983} that a group $\Gamma$ is virtually polycyclic if and only if it is (poly-$\ZZ$)-by-finite.

The \emph{Hirsch length} of a virtually polycyclic group $\Gamma$ is the number of infinite cyclic factors in a subnormal series with cyclic or finite factors.  It is denoted by $\rmh(\Gamma)$.  See \cite[Chapter 1, Part C]{segal1983} for a discussion of the Hirsch length and its properties.  In particular, we will make use of the following properties:
\begin{itemize}
\item If $\Lambda \leq \Gamma$, then $\rmh(\Lambda) \leq \rmh(\Gamma)$.  Equality holds if and only if $[\Gamma : \Lambda] < \infty$.
\item If $\Lambda \unlhd \Gamma$ is normal, then $\rmh(\Gamma) = \rmh(\Lambda) + \rmh(\Gamma/\Lambda)$.
\end{itemize}

\subsection{C*-uniqueness}
\label{sec:cstar-uniqueness}

A group $\Gamma$ is called $\lone$-to-\Cstar-unique or just \Cstar-unique for short if $\lone(\Gamma)$ has a unique \Cstar-norm.  It is clear that every \Cstar-unique group is amenable and that a group $\Gamma$ is \Cstar-unique if and only if $\lone(\Gamma)$ has non-zero intersection with every non-zero ideal of $\Cstar(\Gamma)$.  The following result is a special case of \cite[Satz 2]{boidolleptinschurmannwahle1978}.
\begin{theorem}
  \label{thm:polynomial-growth-group-cstar-unique}
  Every finitely generated group of polynomial growth is \Cstar-unique.
\end{theorem}

In this work we will only need to apply \cref{thm:polynomial-growth-group-cstar-unique} in the special case of finitely generated torsion-free abelian groups, that is groups isomorphic with $\ZZ^n$ for some $n \in \NN$.  For the sake of a self-contained presentation, we give a short and direct proof in this case.

\begin{proof}[Proof of \cref{thm:polynomial-growth-group-cstar-unique} for $\ZZ^n$]
  It suffices to show that $\lone(\ZZ^n) \subseteq \Cstar(\ZZ^n)$ has the ideal intersection property.  Consider the Schwartz algebra
  \begin{gather*}
    \cS(\ZZ^n) = \{f \in \lone(\ZZ^n) \mid \forall k \in \NN: f(x)|x|^k \to 0 \text{ as } x \lra \infty\}
  \end{gather*}
  and the Fourier isomorphism $\cF \colon \Cstar(\ZZ^n) \lra \cont(\bT^n)$.  Then $\cF(\cS(\ZZ^n)) = \cont^\infty(\bT^n)$ is the algebra of smooth functions and it suffices to show that $\cont^\infty(\bT^n) \subseteq \cont(\bT^n)$ has the ideal intersection property. If $I = \{f \in \cont(\bT^n) \mid f|_A \equiv 0\}$ for some proper closed subset $A \subseteq \bT^n$ is an ideal in $\cont(\bT^n)$, then there is a non-zero smooth function $f \in \contc^\infty(\bT^n \setminus A) \subseteq I$. So $I \cap \cont^\infty(\bT^n) \neq 0$.  
\end{proof}

\subsection{C*-simplicity}
\label{sec:cstar-simplicity}

In this section we recall some terminology from the theory of \Cstar-simple groups, and state one result which is needed for our work and can be directly deduced from the literature.  A group $\Gamma$ is called \Cstar-simple if $\Cstarred(\Gamma)$ is simple.  It is clear that every \Cstar-simple group has the \lone-ideal intersection property.

As proven in \cite{breuillardkalantarkennedyozawa14}, a group $\Gamma$ is \Cstar-simple if and only if it acts freely on its Furstenberg boundary $\partial_{\rF} \Gamma$.  We call a stabiliser group of this action a Furstenberg subgroup.  Recall also that the amenable radical $\rR(\Gamma)$ is the largest amenable normal subgroup of $\Gamma$.  Requiring the amenable radical of a group $\Gamma$ to be a Furstenberg subgroup is equivalent to the requirement that the induced action of $\Gamma / \rR(\Gamma)$ is free.  We need the following result that is not explicitly stated in the literature.
\begin{proposition}
  \label{prop:cstar-simple-quotient}
  Let $\Gamma$ be a group whose amenable radical is a Furstenberg subgroup.  Then $\Gamma / \rR(\Gamma)$ is \Cstar-simple.
\end{proposition}
\begin{proof}
  It follows from \cite[Corollary 8.5]{kawabe17} that the \Cstar-algebra generated by the image of the quasi-regular representation with respect to a Furstenberg subgroup is simple.  So by assumption $\Cstarred(\Gamma / \rR(\Gamma)) = \lambda_{\Gamma/\rR(\Gamma)}(\Cstarred(\Gamma))$ is simple.
\end{proof}

\subsection{Twisted C*-dynamical systems}
\label{sec:twisted-cstar-dynamical-systems}

A twisted \Cstar-dynamical system is a tuple $(A, \Gamma, \alpha, \sigma)$ where $A$ is a \Cstar-algebra, $\Gamma$ is a group and $\alpha \colon \Gamma \to \Aut(A)$, $\sigma \colon \Gamma \times \Gamma \to \cU(\rM(A))$ are maps satisfying
\begin{align*}
  \alpha_{g_1} \circ \alpha_{g_2} & = \Ad(\sigma(g_1,g_2)) \circ \alpha_{g_1g_2} \\
  \sigma(g_1,g_2)\sigma(g_1g_2,g_3) & = \alpha_{g_1}(\sigma(g_2, g_3)) \sigma(g_1, g_2 g_3) \\
  \sigma(g_1,e) & = \sigma(e,g_1) = 1
\end{align*}
for all $g_1,g_2,g_3 \in \Gamma$.  The special case of $A = \CC$ corresponds exactly to a group $\Gamma$ with the choice of a 2-cocycle in $\rZ^2(\Gamma, \rS^1)$.

Denoting by $\pi \colon A \to \bo(H)$ the universal representation of $A$, the reduced twisted crossed product associated with $(A, \Gamma, \alpha, \sigma)$ is the \Cstar-subalgebra $A \rtimes_{\alpha, \sigma, \mathrm{red}} \Gamma \subseteq \bo(H \ot \ltwo(\Gamma))$ generated by the elements $\pi_\alpha(a)\lambda_\sigma(g)$ for $a \in A$, $g \in \Gamma$, where $\pi_\alpha \colon A \to \bo(H \ot \ltwo(\Gamma))$ is given by
\begin{gather*}
  \pi_\alpha(a)(\xi \ot \delta_g) = \pi(\alpha_g^{-1}(a))\xi \ot \delta_g
\end{gather*}
and $\lambda_\sigma$ is the twisted regular representation given by
\begin{gather*}
  \lambda_\sigma(\gamma)(\xi \ot \delta_g) = \pi(\sigma((\gamma g)^{-1}, \gamma))\xi \ot \delta_{\gamma g}
  \eqstop
\end{gather*}
Here $a \in A$, $\gamma, g \in \Gamma$ and $\xi \in H$.  If $A = \CC$, then the reduced twisted crossed product \mbox{\Cstar-algebra} is equal to the reduced twisted group \Cstar-algebra where the twist is given by the 2-cocycle associated to the twisted \Cstar-dynamical system.

The $*$-algebra generated by the elements $\pi_\alpha(a)\lambda_\sigma(g)$ for $a \in A$, $g \in \Gamma$ can be equipped with the \lone-norm
\begin{gather*}
  \left \| \sum_{g \in \Gamma} \pi_\alpha(a_g) \lambda_\sigma(g) \right \|_{\lone} = \sum_{g \in \Gamma} \| a_g\| 
  \eqstop
\end{gather*}
Its completion with respect to this norm is the twisted \lone-crossed product $A \rtimes_{\alpha, \sigma, \lone} \Gamma$.

\begin{convention}
  \label{conv:restriction-twisted-dynamical-system}
  Below we will need to consider restrictions of a twisted \Cstar-dynamical system $(A, \Gamma, \alpha, \sigma)$ to subgroups $\Lambda \leq \Gamma$. For notational ease, we will denote the restrictions of $\alpha$ and $\sigma$ by the same symbols. 
\end{convention}

\subsection{Twisted groupoid algebras}
\label{sec:twisted-groupoid-algebras}

For material on {\'e}tale Hausdorff groupoids and groupoid twists we refer the reader to \cite{simsszabowilliams2020}.  We recall the definition of a twist over a groupoid and its convolution algebra, which will be used in this article.
\begin{definition}
  \label{def:groupoid-twist-definition}
  Let $\cG$ be an {\'e}tale groupoid. A \emph{twist} over $\cG$ is a sequence
  \begin{center}
    \begin{tikzcd}
      & \unitspace{\cG} \times \rS^1 \arrow[r, "i"] & \cE                \arrow[r, "q"]       & \cG          
    \end{tikzcd}
  \end{center}
  where $\unitspace{\cG} \times \rS^1$ is the trivial group bundle with fibres $\rS^1$, where $\cE$ is a locally compact Hausdorff groupoid with unit space $i(\unitspace{\cG} \times \{1\})$, and $i$ and $q$ are continuous groupoid homomorphisms that restrict to homeomorphisms of unit spaces, such that
  \begin{itemize}
  \item $i$ is injective,
  \item $\cE$ is a locally trivial $\cG$-bundle, that is for every point $g \in \cG$ there is an open neighbourhood $U$ that is a bisection and on which there exists a continuous section $S \colon U \lra \cE$ satisfying $q \circ S = \id_U$, and such that the map $(g , \mu) \mapsto i(r(g), \mu) S(g)$ is a homeomorphism of $U \times \rS^1$ onto $q^{-1}(U)$
  \item $i(\unitspace{\cG} \times \rS^1)$ is central in $\cE$, that is $i(r(g), \mu) g = g i(s(g), \mu)$ holds for all $g \in \cE$ and $\mu \in \rS^1$, and
  \item $q^{-1}(\unitspace{\cG}) = i(\unitspace{\cG} \times \rS^1)$.
  \end{itemize}
\end{definition}

\begin{notation}
  \label{not:groupoid-twist-conventions}
  A twist as in the definition above will be denoted by $(\cE, i, q)$ or simply by $\cE$ if no confusion is possible. Further, we will frequently identify the unit space of $\cE$ and $\cG$.
\end{notation}
Given a twist $q \colon \cE \thra \cG$ over a locally compact {\'e}tale Hausdorff groupoid $\cG$ we write
  \begin{equation*}
    \cC (\cG , \cE)
    := \{f \in \contc (\cE) \mid f(\mu \cdot g) = \ol{\mu} f(g) \text{ for all $g \in \cE$ and $\mu\in \rS^1$} \}
    \eqcomma
  \end{equation*}
  which becomes a $*$-algebra when equipped with the following convolution product and involution \cite[Proposition 5]{kumjian86}.  We consider the action $\rS^1 \grpaction{} \CC \times \cE$ given by $\mu(z,g) = (\ol{\mu} z, \mu g)$ and let $\CC \times_{\rS^1} \cE$ be the quotient, which is a complex line bundle over $\cG$. It carries a partially defined product (that is, it is a small category) given by $[z_1,g_1][z_2,g_2] = [z_1z_2, g_1g_2]$ for any pair of composable elements $g_1, g_2 \in \cE$. The space $\cC (\cG , \cE)$ is isomorphic with the space of sections $\Gamma(\CC \times_{\rS^1} \cE \to \cG)$ by mapping $f \in \cC (\cG , \cE)$ to the section $q(g) \mapsto [f(g), g]$.  This is well-defined since $(f(\mu g),  \mu g) = (\ol{\mu} f(g), \mu g) = \mu (f(g), g)$ holds. The space $\Gamma(\CC \times_{\rS^1} \cE \to \cG)$ carries the natural involution $f^*(g) = \ol{f}(g^{-1})$ and the convolution product
  \begin{gather*}
    f_1 * f_2(g) = \sum_{g_1 g_2 = g} f_1(g_1) f_2(g_2)
    \eqcomma
  \end{gather*}
for $f,f_1,f_2 \in \Gamma(\CC \times_{\rS^1} \cE)$ and $g \in \cG$.
\begin{remark}
  \label{rem:conventions}
  The attentive reader will have noticed that our conventions for $\cC( \cG, \cE)$ slightly differ from the usual requirement that $f(\mu g) = \mu f(g)$.  This goes hand-in-hand with the divergence from Kumjian's convention $\mu(z,g) = (\mu z, \mu g)$. Our choice of conventions is justified by the following example, which is the basis for understanding the construction presented in \cref{sec:groupoid-presentation}.

  Given a discrete group $\Gamma$ and a cocycle $\sigma \in \rZ^2(\Gamma, \rS^1)$ one associates the central extension $\rS^1 \hra \Gamma \times_\sigma \rS^1 \thra \Gamma$, where the product in $\Gamma \times_\sigma \rS^1$ is given by $(\gamma_1, \mu_1) (\gamma_2, \mu_2) = (\gamma_1\gamma_2, \sigma(\gamma_1, \gamma_2)\mu_1\mu_2)$.  We observe that $(\Gamma, \Gamma \times_\sigma \rS^1)$ is a twisted groupoid and one expects an identification $\cC(\Gamma, \Gamma \times_\sigma \rS^1) \cong \CC[\Gamma, \sigma]$.  This is the case with our conventions, while the usual conventions yield the isomorphism $\cC(\Gamma, \Gamma \times_\sigma \rS^1) \cong \CC[\Gamma, \ol{\sigma}]$. 

  Let us elaborate. For $\gamma \in \Gamma$, we define the section $f_\gamma(\gamma') = \delta_{\gamma, \gamma'} [1, \gamma, 1] \in \CC \times_{\rS^1} (\Gamma \times_\sigma \rS^1)$.  Observe that the function in $\cC(\Gamma, \Gamma \times_\sigma \rS^1)$ associated with it is the unnatural map $(\gamma, \mu) \mapsto \delta_{\gamma, \gamma'}\ol{\mu}$.  We show that the map $\gamma \mapsto f_\gamma$ is $\sigma$-twisted multiplicative.  Indeed, for $\gamma_1, \gamma_2, g \in \Gamma$ we make the calculation
  \begin{align*}
    f_{\gamma_1} * f_{\gamma_2}(g)
    & =
      \sum_{g = g_1 g_2} f_{\gamma_1}(g_1) f_{\gamma_2}(g_2) \\
    & =
      \delta_{g, \gamma_1\gamma_2} [1, \gamma_1, 1] [1, \gamma_2, 1] \\
    & =
      \delta_{g, \gamma_1\gamma_2} [1, \gamma_1\gamma_2, \sigma(\gamma_1, \gamma_2)] \\
    & =
      \delta_{g, \gamma_1\gamma_2} [\sigma(\gamma_1, \gamma_2), \gamma_1\gamma_2, 1] \\
    & =
      \sigma(\gamma_1, \gamma_2) \delta_{g, \gamma_1\gamma_2} [1, \gamma_1\gamma_2, 1] \\
    & =
      \sigma(\gamma_1, \gamma_2) f_{\gamma_1\gamma_2}(g)
      \eqstop
  \end{align*}
  This justifies our conventions sufficiently.
\end{remark}

\begin{convention}
  \label{conv:twist-over-group}
  If $\Gamma$ is a group, then a twist over $\Gamma$ is the same as an extension $1 \to \rS^1 \to E \to \Gamma \to 1$ and hence up to choice of a section $\Gamma \to E$ the same as an element in $\rH^2(\Gamma, \rS^1)$.  In view of \cref{rem:conventions}, in some situations we continue to use the notation $\Cstarred(\Gamma, E)$ for the associated twisted group \Cstar-algebra.
\end{convention}
We will also use the following completions of the twisted groupoid algebra $\cC(\cG, \cE)$.  
\begin{itemize}
\item $\Lone (\cG, \cE)$ denotes its $I$-norm completion, which is a Banach $*$-algebra. Specifically, it is the completion of $\cC (\cG, \cE)$ in the norm 
\begin{align*}
	\Vert f \Vert_I = \max \{ \sup_{x \in \unitspace{\cG}} \{ \sum_{g\in \cE^x} \vert f(g)\vert \}, \sup_{x \in \unitspace{\cG}} \{ \sum_{g\in \cE_x} \vert f(g)\vert \} \}
\end{align*}
for $f \in \cC(\cG,\cE)$.
\item $\Cstar(\cG, \cE)$ denotes the enveloping $\Cstar$-algebra of $\Lone (\cG, \cE)$, and
\item $\Cstarred(\cG, \cE)$ denotes the reduced $\Cstar$-algebra completion of $\Lone(\cG, \cE)$.
\end{itemize}
For a locally compact {\'e}tale Hausdorff groupoid $\cG$ we denote the interior of its isotropy groupoid by $\IsoInt{\cG}$. For $x \in \unitspace{\cG}$, let $\IsoInt{\cG}_x$ be the group appearing in the fibre $x$ of $\IsoInt{\cG}$.  It has been shown in \cite[Corollary 2.11]{armstrong2022-uniqueness} that for a twist $\cE$ over a locally compact {\'e}tale Hausdorff groupoid $\cG$ 
\begin{itemize}
\item the interior of the isotropy subgroupoid $\IsoInt{\cE}$ is a twist over the interior of the isotropy subgroupoid $\IsoInt{\cG}$, and
\item for each $x \in \unitspace{\cG}$, the isotropy group $\IsoInt{\cE}_x$ is a twist over the isotropy group $\IsoInt{\cG}_x$. 
\end{itemize}
We will apply the following result on the ideal intersection property for twisted groupoid \mbox{\Cstar-algebras} associated with a twist over a locally compact {\'e}tale Hausdorff groupoid and its restriction to the interior of the isotropy bundle. We summarise results from \cite[Proposition 6.1 and Theorem 6.3]{armstrong2022-uniqueness}, which generalised previous work in the untwisted case published in \cite{brownnagyreznikoffsimswilliams2016}.
\begin{theorem}[\cite{armstrong2022-uniqueness}]
  \label{thm:essential-monomorphism}
  Let $\cE$ be a twist over a second-countable locally compact {\'e}tale Hausdorff groupoid $\cG$. There is an injective $\ast$-homomorphism $\iota \colon \Cstarred(\IsoInt{\cG}, \IsoInt{\cE}) \to \Cstarred(\cG, \cE)$ such that
  \begin{equation*}
    \iota (f) (g) = 
    \begin{cases}
      f(g) & \text{if } g \in \IsoInt{\cE} \\
      0 & \text{if } g \not \in \IsoInt{\cE}\eqcomma
    \end{cases}
  \end{equation*}
  for all $f \in \cC(\IsoInt{\cG}, \IsoInt{\cE})$ and all $g \in \cE$.  The image of $\iota$ has the ideal intersection property in $\Cstarred(\cG, \cE)$.
\end{theorem}

\section{The \texorpdfstring{${\bm \lone}$}{l1}-ideal intersection property for twisted C*-algebraic dynamical systems: basic results}
\label{sec:lone-ideal-intersection-property-basics}

In general, an inclusion of (\Cstar-)algebras $A \subseteq B$ has the ideal intersection property, if every non-zero ideal in $B$ has non-zero intersection with $A$.  We are interested in this property for the inclusion of $\lone$-algebras into \Cstar-algebras associated with twisted \Cstar-dynamical systems.  In this section we therefore define the ideal intersection property in this generality, before deriving some useful reformulations and results which come in handy later.

\begin{definition}
  \label{def:l1-ideal-intersection-property-crossed-products}
  A twisted \Cstar-dynamical system $(A, \Gamma, \alpha, \sigma)$ is said to have the \lone-ideal intersection property if every non-zero ideal in $A \rtimes_{\alpha, \sigma, \mathrm{red}} \Gamma$ has non-zero intersection with $A \rtimes_{\alpha, \sigma, \lone} \Gamma$.
\end{definition}

In situations, where part of the twisted action is trivial, e.g. for twisted group \Cstar-algebras associated with a pair $(\Gamma, \sigma)$ or untwisted crossed products associated with an action $\Gamma \grpaction{} X$, we simplify notation and say that $(\Gamma, \sigma)$, respectively $\Gamma \grpaction{} X$, has the ideal intersection property.

\begin{remark}
  \label{remark:relation-to-simplicity-and-C-star-uniqueness}
  Let us put the notion introduced in the previous definition into context.
  \begin{itemize}
  \item Every twisted \Cstar-dynamical system with a simple crossed product trivially has the \lone-ideal intersection property.  Such systems arise from \Cstar-simple groups \cite{bryderkennedy2018-twisted}.
  \item For an amenable twisted \Cstar-dynamical system $(A,\Gamma, \alpha, \sigma)$ the \lone-ideal intersection property is equivalent to \Cstar-uniqueness of $A \rtimes_{\alpha, \sigma, \lone} \Gamma$. This can be inferred from the fact that reduced and universal crossed products of such systems coincide combined with \cite[Proposition 2.4]{barnes1983}.  Alternatively, \cref{prop:equivalent-reformulations} below can be employed.
\end{itemize}
\end{remark}
The following reformulation shows that the \lone-ideal intersection property for twisted \mbox{\Cstar-dynamical} systems is a question of minimality of the reduced \Cstar-algebra norm on $A \rtimes_{\lone, \sigma} \Gamma$.  It will be frequently used without further reference.
\begin{proposition}
  \label{prop:equivalent-reformulations}
  Let $(A,\Gamma,\alpha, \sigma)$ denote a twisted $C^*$-dynamical system. The following conditions are equivalent.
  \begin{enumerate}
  \item \label{it:equiv:intersection-property}
    $(A,\Gamma,\alpha, \sigma)$ has the \lone-ideal intersection property.
  \item \label{it:equiv:faithful-homomorphism}
    If a $*$-homomorphism into a \Cstar-algebra $\pi \colon A \rtimes_{\alpha, \sigma, \mathrm{red}} \Gamma \to B$ is injective on $A \rtimes_{\alpha, \sigma, \lone} \Gamma$ then it is injective itself.
  \item \label{it:equiv:minimal-norm}
    The reduced \Cstar-norm on $A \rtimes_{\alpha, \sigma, \lone} \Gamma$ is minimal.    
  \end{enumerate}
\end{proposition}
\begin{proof}
  Suppose that $(A,\Gamma,\alpha, \sigma)$ has the \lone-ideal intersection property and let $\pi \colon A \rtimes_{\alpha, \sigma, \mathrm{red}} \Gamma \to B$ be injective on $A \rtimes_{\alpha, \sigma, \lone} \Gamma$.  Then $\ker \pi \cap A \rtimes_{\alpha, \sigma, \lone} \Gamma = \{0\}$ implies that $\ker \pi = 0$. So $\pi$ itself is injective.

  Assume next that \cref{it:equiv:faithful-homomorphism} holds and let $\nu \colon A \rtimes_{\alpha, \sigma, \lone} \Gamma \to \RR_{\geq 0}$ be a \Cstar-norm dominated by the reduced \Cstar-norm.  Denoting by $B = \ol{A \rtimes_{\alpha, \sigma, \lone} \Gamma}^{\nu}$ the completion, the natural $*$-homomorphism $\pi \colon A \rtimes_{\alpha, \sigma, \mathrm{red}} \Gamma \to B$ is faithful on the $\lone$-crossed product.  The assumption implies that $\pi$ is injective and henceforth isometric.  Thus, $\nu$ is equal to the reduced \Cstar-norm.

  Now assume that \cref{it:equiv:minimal-norm} holds and let $I \unlhd A \rtimes_{\alpha, \sigma, \mathrm{red}} \Gamma$ be a non-zero ideal with quotient map $\pi \colon A \rtimes_{\alpha, \sigma, \mathrm{red}} \Gamma \to B$, which is contractive.  The assumption allows us to infer that $\ker \pi \cap A \rtimes_{\alpha, \sigma, \lone} \Gamma \neq \{0\}$.  Since $I = \ker \pi$ and $I$ was arbitrary, we conclude that $(A,\Gamma,\alpha, \sigma)$ has the \lone-ideal intersection property.
\end{proof}

\begin{remark}
  \label{rem:alekseev-kyed}
  The analogue of the minimality of the reduced \Cstar-norm featuring in \cref{it:equiv:minimal-norm} of \cref{prop:equivalent-reformulations} has previously been introduced for algebraic group rings in \cite{alekseevkyed2019} under the name \mbox{$\Cstar_{\mathrm{r}}$-uniqueness}.

  Let us also remark that an amenable group has the $\lone$-ideal intersection property if and only if it is C*-unique.
\end{remark}
We next show that the \lone-ideal intersection property is closed under directed unions, in the following precise sense.
\begin{proposition}
  \label{prop:lone-ideal-intersection-directed-union}
  Let $(A, \Gamma, \alpha, \sigma)$ be a twisted \Cstar-dynamical system.  Assume that $\Gamma = \bigcup_{i \in I} \Gamma_i$ is a directed union such that $(A, \Gamma_i, \alpha, \sigma)$ has the \lone-ideal intersection property for all $i \in I$.  Then $(A, \Gamma, \alpha, \sigma)$ has the \lone-ideal intersection property.
\end{proposition}
\begin{proof}
  We employ the characterisation in \cref{it:equiv:faithful-homomorphism} of \cref{prop:equivalent-reformulations} of the \lone-ideal intersection property.  Let $\pi \colon A \rtimes_{\alpha, \sigma, \mathrm{red}} \Gamma \to B$ be a $*$-homomorphism whose restriction to the \lone-crossed product is injective.  The assumptions imply that for all $i \in I$ the restriction $\pi|_{A \rtimes_{\alpha, \sigma, \mathrm{red}} \Gamma_i}$ is injective and thus isometric.  Since $\bigcup_{i \in I} A \rtimes_{\alpha, \sigma, \mathrm{red}} \Gamma_i$ is dense in $A \rtimes_{\alpha, \sigma, \mathrm{red}} \Gamma$, the result follows.
\end{proof}

\section{The \texorpdfstring{$\bm \Lone$}{L1}-ideal intersection property for twisted groupoids and their isotropy bundles}
\label{section:groupoid-result-extension}

In this section we prove the \Lone-ideal intersection property for certain twisted groupoids, in the same spirit as \cite{austadortega2022} did for cocycle twisted groupoids.  In view of applications in \cref{sec:main-results}, our statements are established in slightly greater generality.

\begin{definition}
  \label{def:lone-ideal-intersection-property-groupoids}
  A twisted locally compact {\'e}tale Hausdorff groupoid $(\cG, \cE)$ is said to have the \mbox{\Lone-ideal} intersection property if every non-zero ideal of $\Cstarred(\cG, \cE)$ has non-zero intersection with $\Lone(\cG, \cE)$.
\end{definition}
The next result shows that in order to establish the \Lone-ideal intersection property for a twisted groupoid, it suffices to study its isotropy bundle.  It is a direct consequence of Armstrong's results recalled in \cref{sec:twisted-groupoid-algebras}, and its analogue in the context of \Cstar-uniqueness of cocycle twisted groupoid \Cstar-algebras was obtained in \cite[Proposition 3.2]{austadortega2022}.
\begin{proposition}
  \label{prop:isotropy-subgroupoid-sufficient}
  Let $\cG$ be a second-countable locally compact {\'e}tale Hausdorff groupoid and let $\cE$ be a twist over $\cG$. If $(\IsoInt{\cG}, \IsoInt{\cE})$ has the $L^1$-ideal intersection property, then so does $(\cG, \cE)$.
\end{proposition}
\begin{proof}
Let $I \unlhd \Cstarred(\cG, \cE)$ be a non-zero ideal.  Appealing to the work of \cite{armstrong2022-uniqueness} described in \cref{thm:essential-monomorphism} and identifying $\Cstarred(\IsoInt{\cG}, \IsoInt{\cE})$ with its image in $\Cstarred(\cG, \cE)$, we find a that $J = I \cap \Cstarred(\IsoInt{\cG}, \IsoInt{\cE})$ is a non-zero ideal.  By assumption, we can thus conclude that
\begin{gather*}
  0 \neq J \cap \Lone(\IsoInt{\cG}, \IsoInt{\cE}) \subseteq I \cap \Lone(\cG, \cE)
\end{gather*}
which completes the proof.
\end{proof}
In the remainder of this section we aim to prove that if sufficiently many fibres of the isotropy bundle have the $\lone$-ideal intersection property, then the full isotropy bundle has the $\Lone$-ideal intersection property.  Following the same strategy as in \cite{austadortega2022}, we achieve this by decomposing any \Cstar-completion of $\Lone(\IsoInt{\cG}, \IsoInt{\cE})$ as a \Cstar-bundle over $\unitspace{\cG}$.  We will need the following lemma in order to describe the fibres of this bundle. It generalises \cite[Lemma 3.4]{austadortega2022}, but we give a shorter proof applicable in greater generality, which is later needed in the proof of \cref{thm:ideal-intersection-general-form}.  Given a groupoid $\cG$, we call $x \in \unitspace{\cG}$ {\em strongly fixed} if $\cG_x = \IsoInt{\cG}_x$, that is every groupoid element whose source (or range) is $x$ lies in the interior of the isotropy bundle.
\begin{lemma}
  \label{lem:L1-isometric-isomorphism}
  Let $\cE$ be a twist over a locally compact {\'e}tale Hausdorff groupoid $\cG$.  Assume that $x \in \unitspace{\cG}$ is a strongly fixed point and denote by $\res_x \colon \Lone(\cG, \cE) \lra \Lone(\cG_x, \cE_x)$ the restriction map and by $I_x$ its kernel.  Then $\res_x$ is a continuous $*$-homomorphism which induces an isometric $*$-isomorphism between $\Lone(\cG, \cE)/I_x$ and $\Lone(\cG_x, \cE_x)$.  Further, $I_x$ is the ideal generated by $\conto(\unitspace{\cG} \setminus \{x\})$. 
\end{lemma}
\begin{proof}
  It is clear that $\res_x$ is continuous and in order to show that it induces an isometric \mbox{$*$-isomorphism}, it suffices to show that $\res_x|_{\cC(\cG, \cE)}$ factors through to an isometry with dense image.

  We first prove density.  Let $f_x \in \cC(\cG_x, \cE_x)$ be arbitrary. Considering $\cE_x \subseteq \cE$ as a closed subset and making use of local compactness of the latter, Tietze's theorem provides some function $\tilde f_x \in \contc(\cE)$ such that $\tilde f_x|_{\cE_x} = f_x$.  Define
  \begin{gather*}
    f(g) = \int_{\rS^1} \mu \tilde f_x(\mu g) \rmd \mu
    \eqstop
  \end{gather*}
  Then $f \in \cC(\cG, \cE)$ holds thanks to invariance of the Haar measure, and $f|_{\cE_x} = f_x$ by $\rS^1$-equivariance of $f_x$.  This proves density of the image.
    
  Given $f \in \cC(\cG, \cE)$ and $\veps > 0$ there is a neighbourhood $U \subseteq \unitspace{\cG}$ of $x$ such that $\supp f \cap \rms^{-1}(U) \subseteq \IsoInt{\cE}$ and $\big | \|\res_y(f)\| - \|\res_x(f)\| \big | < \veps$ for all $y \in U$.  Since $\unitspace{\cG}$ is locally compact, by Tietze's theorem there is $g \in \cont(\unitspace{\cG})$ with $0 \leq g \leq 1$, $g|_{\unitspace{\cG} \setminus U} \equiv 1$ and $g(x) = 0$.  Then $f * g \in I_x$ and we find that
  \begin{gather*}
    \|f + I_x\|
    \leq
    \|f - f*g\|
    \leq
    \sup_{y \in U} \|\res_y(f)\| 
    \leq \|\res_x(f)\| + \veps
    \eqstop
  \end{gather*}
  Further, 
  \begin{gather*}
    \|\res_x(f)\|
    =
    \inf_{h \in I_x} \|\res_x(f + h)\|
    \leq
    \inf_{h \in I_x} \|f + h\| 
    =
    \|f + I_x\|
    \eqstop
  \end{gather*}
  It remains to show that $I_x$ is equal to the ideal $J$ generated by $\conto(\unitspace{\cG} \setminus \{x\})$ in $\Lone(\cG, \cE)$.  If $f \in I_x$ and $\veps > 0$, there is $\tilde f \in \cC(\cG, \cE)$ such that $\|f - \tilde f\|_I < \veps$.  Thus $\|\res_x(\tilde f)\| < \veps$ and hence we find as above $g \in \conto(\unitspace{\cG} \setminus \{x\})$ such that $\|\tilde f - \tilde f * g\|_I < \veps$.  This implies that $\|f - \tilde f * g\| < 2\veps$.  Since $\tilde f * g \in J$ and $\veps > 0$ was arbitrary, this finishes the proof.
\end{proof}
We are now ready to prove the main result of this section, which generalises \cite[Theorem 3.1]{austadortega2022}.  It is stated and proven in the generality needed for \cref{thm:ideal-intersection-general-form}.  Extending usual conventions and accepting zero-fibres, for a $*$-homomorphism $\conto(X) \to \cZ (\rM(\cA))$, we denote by $\cA_x$ the quotient of $\cA$ by the ideal generated by the image of $\conto(X \setminus \{x\})$.
 
\begin{theorem}
  \label{theorem:fiber-groups-sufficient}  
  Let $\cE$ be a twist over a second-countable locally compact {\'e}tale Hausdorff groupoid $\cG$.  Assume that there is a dense subset $D \subseteq \unitspace{\cG}$ such that $\lone(\IsoInt{\cG}_x, \IsoInt{\cE}_x) \subseteq \Cstarred(\IsoInt{\cG}_x, \IsoInt{\cE}_x)$ has the ideal intersection property for all $x \in D$.  Let $\pi \colon \Cstarred(\cG, \cE) \to \cA$ be a $*$-homomorphism into a \Cstar-algebra which is injective on $\conto(\unitspace{\cG})$, so that $\pi(\Cstarred(\IsoInt{\cG}, \IsoInt{\cE}))$ becomes a  $\conto(\unitspace{\cG})$-algebra. 
  For $x \in \unitspace{\cG}$, consider the $*$-homomorphism
  \begin{gather*}
    \pi_x \colon \Cstarred(\IsoInt{\cG}_x, \IsoInt{\cE}_x) \to \pi(\Cstarred(\IsoInt{\cG}, \IsoInt{\cE}))_x
  \end{gather*}
obtained from $\pi\vert_{\Cstarred(\IsoInt{\cG}, \IsoInt{\cE})}$ after passing to the quotient by the ideals generated by $\conto(\unitspace{\cG} \setminus \{x\})$.  If $\pi_x$ restricts to an injection of $\lone(\IsoInt{\cG}_x, \IsoInt{\cE}_x)$ for all $x \in D$, then $\pi$ is injective.
\end{theorem}
\begin{proof}
  Let $\pi \colon \Cstarred(\cG, \cE) \to \cA$ and $D \subseteq \unitspace{\cG}$ be as in the statement of the theorem.  Without loss of generality, we may assume that $\pi$ is non-degenerate.  By \cref{thm:essential-monomorphism}, it suffices to show that $\pi|_{\Cstarred(\IsoInt{\cG}, \IsoInt{\cE})}$ is injective.  Since $\pi_x|_{\lone(\IsoInt{\cG}_x, \IsoInt{\cE}_x)}$ is injective for all $x \in D$, it is in particular non-zero, so that density of $D \subseteq \unitspace{\cG}$ implies that $\pi|_{\conto(\unitspace{\cG})}$ is injective.  Hence $\cB = \pi(\Cstarred(\IsoInt{\cG}, \IsoInt{\cE}))$ is a $\conto(\unitspace{\cG})$-algebra.  Denote by $B = (\cB_x)_x$ the upper semi-continuous \Cstar-bundle associated with it by \cite[Theorem 2.3]{nilsen1996}, which recovers $\cB$ as the algebra of sections $\cB \cong \Gamma_0(B)$.

  By \cref{lem:L1-isometric-isomorphism}, we obtain the following commutative diagram upon taking quotients by the ideal generated by $\conto(\unitspace{\cG} \setminus \{x\})$ in each algebra of its top row.
    \begin{center}
      \begin{tikzcd}
        \Lone(\IsoInt{\cG}, \IsoInt{\cE}) \ar[r] \ar[d] & \Cstarred(\IsoInt{\cG}, \IsoInt{\cE}) \ar[r, twoheadrightarrow] \ar[d] & \cB \ar[d] \\
        \lone(\IsoInt{\cG}_x, \IsoInt{\cE}_x) \ar[r] & \Cstarred(\IsoInt{\cG}_x, \IsoInt{\cE}_x) \ar[r, twoheadrightarrow, "\pi_x"] & \cB_x
      \end{tikzcd}
    \end{center}
    For $x \in D$, the $*$-homomorphism $\lone(\IsoInt{\cG}_x, \IsoInt{\cE}_x) \to \cB_x$ is injective and $(\IsoInt{\cG}_x, \IsoInt{\cE}_x)$ has the \lone-ideal intersection property.  So $\pi_x$ is an isomorphism of \Cstar-algebras and as such an isometry.  Let now $f \in \Gamma_0(B)$ be an element in the image of $\Lone(\IsoInt{\cG}, \IsoInt{\cE})$.  Then
    \begin{gather*}
      \|f\|_\cB
      =
      \sup_{x \in \unitspace{\cG}} \|f(x)\|_{\cB_x}
      \geq 
      \sup_{x \in D} \|f(x)\|_{\cB_x}
      =
      \sup_{x \in D} \|f(x)\|_{\Cstarred(\IsoInt{\cG}_x, \IsoInt{\cE}_x)}
      =
      \|f\|_{\Cstarred(\IsoInt{\cG}, \IsoInt{\cE})}
    \end{gather*}
    since the regular representations of $(\IsoInt{\cG}, \IsoInt{\cE})$ are continuous by construction \cite[Section 2]{kumjian86}.
\end{proof}

\begin{corollary}
  \label{cor:groupoid-lone-ideal-intersection-property}
  Let $\cE$ be a twist over a second-countable locally compact {\'e}tale Hausdorff groupoid $\cG$.  Assume that there is a dense subset $D \subseteq \unitspace{\cG}$ such that $(\IsoInt{\cG}_x, \IsoInt{\cE}_x)$ has the \lone-ideal intersection property for all $x \in D$.  Then $(\cG, \cE)$ has the \Lone-ideal intersection property.
\end{corollary}
\begin{proof}
  Let $\pi\colon \Cstarred(\cG, \cE) \to \cA$ be a $*$-homomorphism that is injective on $\Lone(\cG, \cE)$ and write $\cB = \pi(\Cstarred(\IsoInt{\cG}, \IsoInt{\cE}))$.  In order to prove injectivity of $\pi$, by~\cref{theorem:fiber-groups-sufficient}, it suffices to check that $\pi_x\colon \Cstarred(\IsoInt{\cG}_x, \IsoInt{\cE}_x) \to \cB_x$ is injective when restricted to $\lone(\IsoInt{\cG}_x, \IsoInt{\cE}_x)$ for all $x \in \unitspace{\cG}$.  By \cref{lem:L1-isometric-isomorphism}, taking the quotient by the ideal generated by $\conto(\unitspace{\cG} \setminus \{x\})$ in the inclusion $\Lone(\cG, \cE) \hra \cB$, we indeed obtain the desired inclusion $\lone(\IsoInt{\cG}_x, \IsoInt{\cE}_x) \hra \cB_x$, which finishes the proof.
\end{proof}

\section{Groupoid C*-algebras from abelian normal subgroups}
\label{sec:groupoid-presentation}

In this section we describe a twisted groupoid associated with an inclusion of a normal abelian subgroup into a discrete group endowed with an $\rS^1$-valued 2-cocycle.  This construction should be folklore, but has not been presented explicitly to our knowledge.

\begin{definition}
  \label{def:admissible-cocycle}
  Let $A \unlhd \Gamma$ be a normal abelian subgroup of a discrete group.  A cocycle $\sigma \in \rZ^2(\Gamma, \rS^1)$ is $A$-admissible if it satisfies
  \begin{itemize}
  \item $\sigma|_{A \times A} \equiv 1$, and
  \item $\sigma(\gamma, a)\sigma(\gamma a , \gamma^{-1}) = 1 = \sigma(a , \gamma^{-1})\sigma(\gamma , a \gamma^{-1})$ for all $\gamma \in \Gamma$ and $a \in A$.
  \end{itemize}
\end{definition}
Let $A \unlhd \Gamma$ and $\sigma$ be as above.  Write $\Lambda = \Gamma/A$ and consider the action $\Lambda \grpaction{\alpha} A$ given by $\alpha_\lambda(a) = \gamma a \gamma^{-1}$ for $\gamma A = \lambda$.  Since $A$ is abelian, this is well-defined.  Denote by $\cG = \Lambda \ltimes \hat A$ the transformation groupoid associated with the dual action of $\alpha$.  Further, let $\Gamma \ltimes_\sigma (\rS^1 \times \hat A)$ be the twisted transformation groupoid whose product is given by
\begin{gather*}
  ( \gamma_1, \mu_1, \gamma_2\chi) ( \gamma_2, \mu_2, \chi)
  =
  (\gamma_1 \gamma_2, \mu_1 \mu_2 \sigma(\gamma_1, \gamma_2), \chi)
\end{gather*}
for $\gamma_1, \gamma_2 \in \Gamma$, $\mu_1, \mu_2 \in \rS^1$ and $\chi \in \hat A$ and consider
\begin{gather*}
  \cN = \{ (a^{-1}, \chi(a), \chi) \mid  a \in A, \chi \in \hat A \} \subseteq \Gamma \ltimes_\sigma (\rS^1 \times \hat A)
  \eqstop
\end{gather*}
The following lemma describes a twisted groupoid associated to the tuple $(\Gamma, A, \sigma)$.
\begin{lemma}
  \label{lem:normal-subgroupoid}
  The set $\cN \subseteq \Gamma \ltimes_\sigma (\rS^1 \times \hat A)$ is a closed normal subgroupoid.  Further, $\Gamma \ltimes_\sigma (\rS^1 \times \hat A) / \cN$ is a twist over $\cG$.
\end{lemma}
\begin{proof}
  It follows from the fact that evaluation of characters in $\hat A$ is continuous, that $\cN$ is closed. Further, it is multiplicatively closed since $\sigma|_{A \times A} \equiv 1$ and the calculation
  \begin{gather*}
    (a^{-1}, \chi(a), \chi)^{-1}
    =
    (a, \ol{\chi(a)} \sigma(a^{-1}, a), \chi)
    =
    (a, \chi(a^{-1}), \chi)
  \end{gather*}
  for $a \in A$ and $\chi \in \hat A$ shows that $\cN$ is also closed under inverses. So it is a closed subgroupoid of $\Gamma \ltimes_\sigma (\rS^1 \times \hat A)$.  We next check normality of $\cN$.  Thanks to centrality of $\rS^1$ it suffices to observe for $\gamma \in \Gamma$, $a \in A$ and $\chi \in \hat A$ that
  \begin{align*}
    (\gamma, 1, \chi)(a^{-1}, \chi(a), \chi)(\gamma^{-1}, 1, \gamma\chi)
    & =
    (\gamma a^{-1}\gamma^{-1}, \chi(a) \sigma(\gamma, a^{-1})\sigma(\gamma a^{-1}, \gamma^{-1}), \gamma \chi) \\
    & =
    (\gamma a^{-1} \gamma^{-1}, \gamma \chi(\gamma a \gamma^{-1})), \gamma \chi)
    \eqstop
  \end{align*}

  We now want to show that the quotient $\cE = \Gamma \ltimes_\sigma (\rS^1 \times \hat A) / \cN$ is a twist over $\cG$. The inclusion $\{e\} \times \rS^1 \times \hat A \subseteq \Gamma \ltimes_\sigma (\rS^1 \times \hat A)$ descends to an inclusion $i\colon \rS^1 \times \hat A \lra \cE$ since $\cN \cap (\{e\} \times \rS^1 \times \hat A) = \{(e,1)\} \times \hat A$.  Further, the projection onto the first and the last component $\Gamma \ltimes_\sigma (\rS^1 \times \hat A) \lra \Gamma \times \hat A$ induces a continuous quotient map $q\colon \cE \lra \Gamma/A \ltimes \hat A = \cG$.  It is clear that $i(\rS^1 \times \hat A)$ is central in $\cE$ and that $q^{-1}(\{eA\} \times \hat A) = i(\rS^1 \times \hat A)$.  What remains to be shown is that $\cE$ is locally trivial.  Let $(\gamma A, \chi_0) \in \cG$ and consider the open bisection $U = \{\gamma A\} \times \hat A$.  The map $S \colon U \lra \cE \colon (\gamma A, \chi) \mapsto [\gamma, 1, \chi]$ is continuous and satisfies $q \circ S = \id_U$.  Further,
  \begin{align*}
    q^{-1}(U) & = \{ [\gamma a, \mu, \chi] \in \cE \mid a \in A, \mu \in \rS^1, \chi \in \hat A\} \\
              & = \{ [\gamma, \mu, \chi] \in \cE \mid \mu \in \rS^1, \chi \in \hat A\}
  \end{align*}
is naturally isomorphic with $\rS^1 \times U$.
\end{proof}
Let us introduce some notation in order to refer to the twisted groupoid just constructed.
\begin{definition}
  \label{def:groupoid-decomposition}
  Given a group $\Gamma$ with a normal abelian subgroup $A$ and an $A$-admissible 2-cocycle $\sigma \in \rZ^2(\Gamma, \rS^1)$, we denote the associated twisted groupoid by
  \begin{align*}
    \cG(\Gamma, A, \sigma ) & = \Gamma/A \ltimes \hat A \\
    \cE(\Gamma, A, \sigma) & = \Gamma \ltimes_\sigma (\rS^1 \times \hat A) / \{ (a^{-1}, \chi(a), \chi) \mid a \in A, \chi \in \hat A\}
                     \eqstop
  \end{align*}
\end{definition}
We next identify the twisted group algebras associated to $(\Gamma, \sigma)$ with the twisted groupoid algebra associated with a normal abelian subgroup $A \unlhd \Gamma$ for which $\sigma$ is admissible.  This result generalises the identification described in \cref{rem:conventions}.
\begin{proposition}
  \label{prop:identification-algebras-groupoid-decomposition}
  Let $A \unlhd \Gamma$ be an abelian normal subgroup of a discrete group and $\sigma \in \rZ^2(\Gamma, \rS^1)$ an $A$-admissible cocycle.  Let $(\cG, \cE) = (\cG(\Gamma, A, \sigma), \cE(\Gamma, A, \sigma))$ be the associated twisted groupoid and write elements of $\CC \times_{\rS^1} \cE$ as equivalence classes $[z, \gamma, \mu, \chi]$ with $z \in \CC$, $\gamma \in \Gamma$, $\mu \in \rS^1$ and $\chi \in \hat A$.  Given $\gamma \in \Gamma$ define the following section of $\CC \times_{\rS^1} \cE \thra \cG$:
  \begin{gather*}
    f_\gamma(gA, \chi) = 
    \begin{cases}
      [1, \gamma, 1, \chi] & \text{if } gA = \gamma A \\
      0 & \text{otherwise.}
    \end{cases}
  \end{gather*}
  Then the map $\gamma \mapsto f_\gamma$
  \begin{enumerate}
  \item extends to a contractive embedding $\lone(\Gamma, \sigma) \hra \Lone(\cG, \cE)$, which
  \item extends to an isomorphism $\Cstarred(\Gamma, \sigma) \hra \Cstarred(\cG, \cE)$.
  \end{enumerate}
\end{proposition}
\begin{proof}  
  We first show that the map $\gamma \to f_\gamma$ is $\sigma$-twisted multiplicative.  For $\gamma_1, \gamma_2, g \in \Gamma$ and $\chi \in \hat A$, we find that
  \begin{align*}
    f_{\gamma_1} * f_{\gamma_2} (g A, \chi)
    & =
    \sum_{(g_1A) (g_2 A) = g A} f_{\gamma_1}(g_1A, g_2 \chi) f_{\gamma_2}(g_2A, \chi) \\
    & =
      \sum_{\substack{(g_1 A) (g_2 A) = g A \\ g_1 A = \gamma_1 A, \, g_2 A = \gamma_2 A}}
    [1, g_1, 1,  g_2\chi][1, g_2, 1, \chi] \\
    & =
      \delta_{gA, \gamma_1\gamma_2A} [1, \gamma_1, 1, \gamma_2\chi][1, \gamma_2, 1, \chi] \\
    & =
      \delta_{gA, \gamma_1\gamma_2A} [1, \gamma_1\gamma_2, \sigma(\gamma_1, \gamma_2), \chi] \\
    & =
      \sigma(\gamma_1, \gamma_2) f_{\gamma_1\gamma_2}(gA, \chi)
      \eqstop
  \end{align*}
  Since $f_e$ is the neutral element for the convolution product, this shows that the map $\gamma \mapsto f_\gamma$ extends to a unital $*$-homomorphism $\CC[\Gamma, \sigma] \to \Lone(\cG, \cE)$.

  We next show that this $*$-homomorphism extends to a contraction $\lone(\Gamma, \sigma) \to \Lone(\cG, \cE)$.  To this end, we need to identify the functions $\tilde f_\gamma \in \cC(\cG, \cE)$ associated with $f_\gamma$.  We claim that
  \begin{gather*}
    \tilde f_\gamma([g, \mu, \chi]) =
    \begin{cases}
      \ol{\mu} \chi(g^{-1}\gamma) & \text{if } \gamma A = g A \\
      0 & \text{otherwise.}
    \end{cases}
  \end{gather*}
  Indeed, for $\gamma A = gA$, $\mu \in \rS^1$ and $\chi \in \hat A$ we calculate
  \begin{align*}
    [\ol{\mu}\chi(g^{-1}\gamma), g, \mu,  \chi]
    =
    [\chi(g^{-1}\gamma), \gamma \gamma^{-1} g, 1, \chi]
    =
    [\chi(g^{-1}\gamma), \gamma,  \chi(\gamma^{-1} g), \chi] 
     =
    [1, \gamma, 1, \chi]
    \eqstop
  \end{align*}
  Take now $\sum_{\gamma \in \Gamma} c_\gamma u_\gamma \in \CC[\Gamma, \sigma]$.  Then
  \begin{align*}
    \sup_{\chi \in \hat A} \|\sum_{\gamma \in \Gamma} c_\gamma \tilde f_\gamma \|_{\lone(\cG_\chi)}
    & =
      \sup_{\chi \in \hat A} \sum_{gA \in \Gamma/A} \left | \sum_{\gamma \in \Gamma} c_\gamma \tilde f_\gamma([g, 1, \chi]) \right | \\
    & =
      \sup_{\chi \in \hat A} \sum_{gA \in \Gamma/A} \left | \sum_{\gamma \in gA} c_\gamma \chi(\gamma^{-1}g) \right | \\
    & \leq
      \sum_\gamma |c_\gamma|
      \eqstop
  \end{align*}
  Similarly, we obtain that
  \begin{align*}
    \sup_{\chi \in \hat A} \|\sum_{\gamma \in \Gamma} c_\gamma \tilde f_\gamma \|_{\lone(\cG^\chi)}
    & =
      \sup_{\chi \in \hat A} \sum_{gA \in \Gamma/A} \left | \sum_{\gamma \in \Gamma} c_\gamma \tilde f_\gamma([g, 1, g^{-1}\chi]) \right | \\
    & =
      \sup_{\chi \in \hat A} \sum_{gA \in \Gamma/A} \left | \sum_{\gamma \in gA} c_\gamma \chi(g \gamma^{-1}) \right | \\
    & \leq
      \sum_\gamma |c_\gamma|
      \eqstop
  \end{align*}
  Together, these calculations show that $\|\sum_\gamma c_\gamma \tilde f_\gamma\|_{\mathrm{I}} \leq \|\sum_\gamma c_\gamma u_\gamma\|_{\lone(\Gamma)}$.  So indeed, we obtain a contraction $\lone(\Gamma, \sigma) \to \Lone(\cG, \cE)$.

  We now show that the contraction above extends to a $*$-isomorphism $\Cstarred(\Gamma, \sigma) \cong \Cstarred(\cG, \cE)$. This will imply in particular that the map $\lone(\Gamma, \sigma) \to \Lone(\cG, \cE)$ is injective.  Consider the conditional expectation $\rE \colon \Cstarred(\cG, \cE) \to \cont(\hat A)$ given by restriction of functions in $\cC(\cG, \cE)$. Further, denote by $\int \rmd \chi$ the Haar integral on $\hat A$.  We observe that for every $\gamma \in \Gamma$, we have
  \begin{align*}
    \left (\int \rmd \chi \circ \rE \right )(f_\gamma)
    & =
      \int_{\hat A} \tilde f_\gamma([e, 1, \chi]) \rmd \chi \\
    & =
      \begin{cases}
        \int_{\hat A} \chi(\gamma) \rmd \chi & \text{if } \gamma \in A  \\
        0 & \text{otherwise}
      \end{cases} \\
    & =
      \begin{cases}
        1 & \text{if } \gamma = e  \\
        0 & \text{otherwise.}
      \end{cases}
  \end{align*}
  This shows that we obtain an isometric $*$-homomorphism $\Cstarred(\Gamma, \sigma) \to \Cstarred(\cG, \cE)$ and it remains to argue that it has dense image. To this end it suffices to show that for every $gA \in \Gamma/A$ and every section $f \colon \cG \to \CC \times_{\rS^1} \cE$ supported on $\{gA\} \times \hat A$ lies in the image of $\Cstarred(\Gamma, \sigma)$.  Let $f_{\hat A} \colon \hat A \to \CC$ be the unique continuous function such that $f(gA, \chi) = [f_{\hat A}(\chi),  g, 1, \chi]$ for all $\chi \in \hat A$.  We can identify $f_{\hat A}$ with an element in $\cC(\cG, \cE)$, and find that $f = f_g * f_{\hat A}$, which finishes the proof. 
\end{proof}
Let us next describe the isotropy groups and the associated twists. 
We first recall the following definition.
\begin{definition}
	Given a group $\Gamma$ and a group action $\Gamma \grpaction{} X$, the \emph{neighbourhood stabiliser of $x$ in $\Gamma$} is the subgroup $\Gamma^\circ_x = \{\gamma \in \Gamma \mid \exists x \in U \text{ open}\colon \gamma|_U = \id_U\}$.
\end{definition}

\begin{proposition}
  \label{prop:groupoid-decomposition-stabilisers}
  Let $A \unlhd \Gamma$ be an abelian normal subgroup and $\sigma \in \rZ^2(\Gamma, \rS^1)$ an $A$-admissible cocycle.  Let $(\cG, \cE)$ be the associated twisted groupoid.  Then the fibre of $(\IsoInt{\cG}, \IsoInt{\cE})$ at $\chi \in \hat A$ is given by the twist $\Gamma^\circ_\chi \times_\sigma \rS^1/N \to \Gamma^\circ_\chi/A$ obtained from $\Gamma^\circ_\chi \times_\sigma \rS^1 \to \Gamma^\circ_\chi/A$ by dividing out the normal subgroup $N = \llangle (a, \ol{\chi(a)}) \mid a \in A \rrangle \unlhd \Gamma^\circ_\chi \times_\sigma \rS^1$.

  Furthermore, given a section $s \colon \Gamma^\circ_\chi/A \to \Gamma^\circ_\chi$ and the associated 2-cocycle $\rho \in \rZ^2(\Gamma^\circ_\chi/A, A)$, we define a section $\tilde s \colon \Gamma^\circ_\chi/A \to \Gamma^\circ_\chi \times_\sigma \rS^1/N$ by $\tilde s(h) = [s(h), 1]$.  Then the associated $\rS^1$-valued 2-cocycle is $(\chi \circ \rho) \cdot (\sigma \circ (s \times s))$.  
\end{proposition}
\begin{proof}
  It is clear that $\IsoInt{\cG}_\chi = \Gamma^\circ_\chi/A$.  We can thus calculate the fibre
  \begin{gather*}
    \IsoInt{\cE}_\chi
    =
    \{[\gamma, \mu, \chi] \mid \gamma \in \Gamma^\circ_\chi, \mu \in \rS^1\}
    \cong
    \Gamma^\circ_\chi \times_\sigma \rS^1/N
    \eqstop
\end{gather*}
Now fix a section $s \colon \Gamma^\circ_\chi/A \to \Gamma^\circ_\chi$ and define $\tilde s(h) = [s(h), 1]$ as in the statement of the result.  For $h_1, h_2 \in \Gamma^\circ_\chi/A$, using the fact that $\chi$ is fixed by $\Gamma^\circ_\chi$, we find that
\begin{align*}
  \tilde s(h_1) \tilde s(h_2)
  & =
    [s(h_1), 1] [s(h_2), 1] \\
  & =
    [\rho(h_1,h_2) s(h_1h_2), \sigma(s(h_1), s(h_2))] \\
  & =
    [s(h_1h_2) (s(h_1h_2)^{-1}\rho(h_1,h_2)s(h_1h_2)), \sigma(s(h_1), s(h_2))] \\
  & =
    [s(h_1h_2) , (\chi \circ \rho(h_1,h_2)) \cdot (\sigma(s(h_1), s(h_2)))] \\
  & =
    (\chi \circ \rho(h_1,h_2)) \cdot (\sigma(s(h_1), s(h_2))) \tilde s(h_1h_2)
    \eqstop
\end{align*}
This shows that $(\chi \circ \rho) \cdot (\sigma \circ (s \times s))$ is indeed a 2-cocycle and that it is the extension cocycle associated with $\tilde s$.
\end{proof}

\section{Proof of the main results}
\label{sec:main-results}

In this section we prove all main results described in the introduction.  We start with three lemmas, which will be used in the proof of \cref{thm:ideal-intersection-general-form}.

\begin{lemma}
  \label{lem:update-crossed-products}
  Let $\Gamma$ be a group whose subgroups all have the $\lone$-ideal intersection property.  Then any action of $\Gamma$ on a locally compact Hausdorff space has the \lone-ideal intersection property.
\end{lemma}
\begin{proof}
  Let $\pi \colon \conto(X) \redtimes \Gamma \to \cA$ be a *-homomorphism such that $\pi|_{\conto(X) \rtimes_{\lone} \Gamma}$ is injective.  Consider the transformation groupoid $\cG = \Gamma \ltimes X$ and the identification $\conto(X) \redtimes \Gamma \cong \Cstarred(\cG)$.  In order to deduce injectivity of $\pi$, by \cref{theorem:fiber-groups-sufficient} it suffices to show that for all $x \in X$ the restriction of $\pi_x \colon \Cstarred(\IsoInt{\cG})_x \to \pi(\Cstarred(\IsoInt{\cG}))_x$ to $\Lone(\IsoInt{\cG})_x$ is injective.  Denote by $\Gamma^\circ_x = \{g \in \Gamma \mid \exists x \in U \text{ open} \colon g|_U = \id_U\}$ the neighbourhood stabiliser of $x$ and let $\cB = \pi(\conto(X) \redtimes \Gamma^\circ_x)$.  Using \cref{lem:L1-isometric-isomorphism}, we obtain the following commutative diagram, where the top row is divided by the respective ideals generated by $\conto(X \setminus \{x\})$.
  \begin{center}
  \begin{tikzcd}
    & \conto(X) \rtimes_{\lone} \Gamma^\circ_x \arrow[r, hookrightarrow, "\pi"] \arrow[d, twoheadrightarrow] & \cB \arrow[d, twoheadrightarrow] \\
    \Lone(\IsoInt{\cG})_x \arrow[r, "\cong"] \arrow[rrr, bend right, "\pi_x"] & \lone(\Gamma^\circ_x) \arrow[r, hookrightarrow] & \cB_x \arrow[r, "\cong"] & \pi(\Cstarred(\IsoInt{\cG}))_x
  \end{tikzcd}
\end{center}
This shows that \cref{theorem:fiber-groups-sufficient} can be applied and finishes the proof.
\end{proof}
We remark that an application of \cref{cor:groupoid-lone-ideal-intersection-property} to the transformation groupoid $\Gamma \ltimes X$ only shows that the groupoid $\rL^1$-algebra has the ideal intersection property, but not the $\lone$-crossed product, which only admits a contractive embedding $\conto(X) \rtimes_{\lone} \Gamma \hra \Lone(\Gamma \ltimes X)$, but is not isomorphic to the groupoid algebra.

\begin{lemma}
  \label{lem:finite-by-cstar-simple}
  Let $\Gamma$ be finite-by-(\Cstar-simple) and $\sigma \in \rZ^2(\Gamma, \rS^1)$.  Then $(\Gamma, \sigma)$ satisfies the $\lone$-ideal intersection property.
\end{lemma}
\begin{proof}
    Let $F \unlhd \Gamma$ be a finite normal subgroup such that $\Lambda = \Gamma/F$ is \Cstar-simple and let $\sigma \in \rZ^2(\Gamma, \rS^1)$.  After a choice of section $\rms \colon \Lambda \to \Gamma$ satisfying $\rms(e) = e$, we infer from \cite[Theorem 4.1]{packerraeburn89} that $\Cstarred(\Gamma, \sigma) \cong \CC[F, \sigma] \rtimes_{\alpha, \rho, \mathrm{red}} \Lambda$, where the twisted crossed product is defined with respect to the maps
  \begin{align*}
    \alpha & \colon \Lambda \to \Aut(\CC[F, \sigma]) \colon \\
           & \qquad \alpha_h(u_f) = \sigma(\rms(h), f) \sigma(\rms(h)  f \rms(h)^{-1}, \rms(h)) u_{\rms(h)f\rms(h)^{-1}} \\
    \rho & \colon \Lambda \times \Lambda \to \cU(\CC[F, \sigma]) \colon \\
           & \qquad \rho(h_1,h_2) = \sigma(\rms(h_1), \rms(h_2)) \ol{\sigma(\rms(h_1)\rms(h_2)\rms(h_1h_2)^{-1}, \rms(h_1h_2))} u_{\rms(h_1)\rms(h_2)\rms(h_1h_2)^{-1}}
    \eqstop
  \end{align*}
  Inspection of the proof of \cite[Theorem 4.1]{packerraeburn89} shows that moreover the inclusion $\lone(\Gamma, \sigma) \subseteq \Cstarred(\Gamma, \sigma)$ is isomorphic with the inclusion of twisted crossed products $\CC[F, \sigma] \rtimes_{\alpha, \rho, \lone} \Lambda \subseteq \CC[F, \sigma] \rtimes_{\alpha, \rho, \mathrm{red}} \Lambda$.  So it suffices to show that $\CC[F, \sigma] \subseteq \CC[F, \sigma] \rtimes_{\alpha, \rho, \mathrm{red}} \Lambda$ satisfies the ideal intersection property.
  
  Since $\CC[F, \sigma]$ is finite dimensional, it is a multi-matrix algebra and hence the twisted \mbox{\Cstar-dynamical} system $(\CC[F, \sigma], \Lambda, \alpha, \rho)$ decomposes as a direct sum of $\Lambda$-simple dynamical systems, say $\CC[F, \sigma] \cong \bigoplus_{i = 1}^n A_i$.  We can apply \cite[Corollary 4.4]{bryderkennedy2018-twisted} to infer that $A_i \rtimes_{\alpha, \rho,\mathrm{red}} \Lambda$ is simple.  So ideals of $\CC[F, \sigma] \rtimes_{\alpha, \rho, \mathrm{red}} \Lambda$ are precisely of the form 
  \begin{gather*}
    I =
    \bigoplus_{i \in S}
    \left ( A_i \rtimes_{\alpha, \rho,\mathrm{red}} \Lambda \right )
  \end{gather*}
  for some subset $S \subseteq \{1, \dotsc, n\}$.  If $I \cap \CC[F, \sigma] = \{0\}$, then $S = \emptyset$ follows, which in turn implies $I = 0$.  This finishes the proof of the lemma.
\end{proof}
For the next lemma recall the notion of admissible cocycles from \cref{def:admissible-cocycle}.
\begin{lemma}
  \label{lem:trivialising-cocycle}
  Let $A \unlhd \Gamma$ be a normal finitely generated abelian subgroup and let $\sigma \in \rZ^2(\Gamma, \ZZ/n\ZZ)$.  There is a finite index characteristic subgroup $B \leq A$ and a $B$-admissible cocycle $\rho \in \rZ^2(\Gamma, \ZZ/n\ZZ)$ equivalent to $\sigma$.
\end{lemma}
\begin{proof}
  Denote by $o = |\mathrm{Tors}(A)|$ the order of the torsion subgroup of $A$ and let $B \leq A$ be the intersection of all its finite index subgroups of index $o$.  Then $B$ has finite index, since $A$ is finitely generated, and further $B$ is characteristic in $A$.  Also $B$ is a finitely generated torsion-free abelian group so that the isomorphism $\rH_2(B) \cong B \wedge B$ together with the universal coefficient theorem in cohomology imply that $\sigma|_{B \times B} \in \rZ^2(B, \ZZ/n\ZZ)$ is equivalent to a bicharacter.  Specifically, there is a map $\vphi \colon B \to \ZZ/n\ZZ$ such that $(b_1,b_2) \mapsto \sigma(b_1,b_2) - \vphi(b_1b_2) + \vphi(b_1) + \vphi(b_2)$ is a bicharacter.  Extending $\vphi$ to a map $\tilde \vphi \colon \Gamma \to \ZZ/n\ZZ$, we may replace $\sigma$ by an equivalent 2-cocycle $\rho$ satisfying
  \begin{gather*}
    \rho(\gamma_1, \gamma_2)
    =
    \sigma(\gamma_1, \gamma_2) - \tilde \vphi(\gamma_1\gamma_2) + \tilde \vphi(\gamma_1) + \tilde \vphi(\gamma_2)
    \qquad
    \text{ for all } \gamma_1,\gamma_2 \in \Gamma
    \eqstop
  \end{gather*}
  Let $i$ be the index of the finite index subgroup $\{b \in B \mid \forall b' \in B: \sigma(b,b') = \sigma(b',b) = 0\} \leq B$.  We denote by $C$ the intersection of all subgroups of $B$ with index $i$, which is of finite index and characteristic in $B$.  Consider now the central extension
  \begin{gather*}
    \ZZ/n\ZZ \hra \tilde \Gamma \thra \Gamma
  \end{gather*}
  associated with $\rho$.  Since $C$ is torsion-free, its preimage in $\tilde \Gamma$ is isomorphic with $C \oplus \ZZ/n\ZZ$ in such a way that the action of $\Gamma$ on it is given by $\alpha_\gamma(c,k) = (\gamma c \gamma^{-1},  \sigma(\gamma, c) + \sigma(\gamma c, \gamma^{-1}))$ for all $\gamma \in \Gamma$, $c \in C$, $k \in \ZZ/n\ZZ$.  In particular, since $\ZZ/n\ZZ$ has exponent $n$, we find that
  \begin{gather*}
    (\gamma c^n \gamma^{-1}, \sigma(\gamma, c^n) + \sigma(\gamma c^n, \gamma^{-1}))
    =
    \alpha_\gamma((c^n, 0))
    =
    \alpha_\gamma((c,0))^n
    =
    ((\gamma c \gamma^{-1})^n, 0)
    \eqstop
  \end{gather*}
  This implies that the subgroup $D = \{ c^n \mid c \in C \} \leq C$ satisfies $\rho(\gamma,d) + \rho(\gamma d, \gamma^{-1}) = 0$ for all $\gamma \in \Gamma$ and $d \in D$.  By definition $D \leq C$ is characteristic. Further it has finite index, because $C$ is finitely generated abelian.
\end{proof}
The next definition describes the groups for which we prove the \lone-ideal intersection property in the subsequent theorem.
\begin{definition}
  \label{def:polycyclically-bounded-groups}
  We denote by $\cU$ the class of all discrete groups $\Gamma$ such that the following three conditions hold for every finitely generated subgroup of $\Lambda \leq \Gamma$:
  \begin{itemize}
  \item for every subgroup of $\Lambda$, its amenable radical is a Furstenberg subgroup,
  \item every amenable subgroup of $\Lambda$ is virtually solvable, and
  \item there is $l \in \NN$ such that every solvable subgroup of $\Lambda$ is polycyclic of Hirsch length at most $l$.
  \end{itemize}
\end{definition}
We are now ready to prove the main theorem of this work.
\begin{theorem}
  \label{thm:ideal-intersection-general-form}
  Let $\Gamma$ be a group from the class $\cU$, let $X$ be a locally compact Hausdorff space and let $\Gamma \grpaction{} X$ be an action by homeomorphisms.  Then $\Gamma \grpaction{} X$ has the \lone-ideal intersection property.
\end{theorem}
\begin{proof}
  By \cref{lem:update-crossed-products}, it suffices to consider the case where $X$ is a point, that is group \mbox{\Cstar-algebras}.  Further, thanks to \cref{prop:lone-ideal-intersection-directed-union}, it suffices to prove the statement for groups with a uniform bound on the Hirsch length of their polycyclic subgroups.

  We prove the following statement: for $\Gamma$ from the class $\cU$ and $\sigma \in \rZ^2(\Gamma, \rS^1)$ a 2-cocycle taking values in a finite subgroup of $\rS^1$, the twisted group \Cstar-algebras $\Cstarred(\Gamma, \sigma)$ has the \lone-ideal intersection property.  The statement is clear for finite groups.  For an induction, fix $l \geq 1$ and assume that the $\lone$-ideal intersection property holds for all 2-cocycles with values in a finite subgroup of $\rS^1$ on groups in $\cU$ whose polycyclic subgroups all have Hirsch length at most $l - 1$.  Let $\Gamma$ be a group in $\cU$ all whose polycyclic subgroups have Hirsch length at most $l \geq 1$, and let $\sigma \in \rZ^2(\Gamma, \rS^1)$ be a cocycle with values in a finite subgroup of $\rS^1$, say $\ZZ/n\ZZ \subseteq \rS^1$.  Thanks to \cref{prop:lone-ideal-intersection-directed-union}, we may assume that $\Gamma$ is finitely generated.  Let $\nu \colon \lone(\Gamma, \sigma) \to \RR_{\geq 0}$ be a \Cstar-norm dominated by $\| \cdot \|_{\mathrm{red}}$.  We denote by $\cA$ the completion of $\lone(\Gamma, \sigma)$ with respect to $\nu$.  Since $\Gamma \in \cU$ and $\Gamma$ is finitely generated, its amenable radical $\rR(\Gamma)$ is virtually polycyclic.  If it is finite, we infer from \cref{prop:cstar-simple-quotient} that $\Gamma$ itself is finite-by-(\Cstar-simple).  So \cref{lem:finite-by-cstar-simple} can be applied.  Otherwise, its maximal polycyclic subgroup  $\Lambda \leq \rR(\Gamma)$ is infinite.  Let $d$ be the derived length of $\Lambda$ and observe that $\Lambda^{(d-1)}$ is a finitely generated abelian group.  By \cref{lem:trivialising-cocycle}, there is a finite index characteristic subgroup $A \leq \Lambda^{(d-1)}$ and an $A$-admissible cocycle $\rho \in \rZ^2(\Gamma, \ZZ/n\ZZ)$ equivalent to $\sigma$.  Since equivalence of cocycles preserves the isomorphism class of twisted group algebras, we may assume that $\rho = \sigma$.

  Observe that all the inclusions $A \leq \Lambda^{(d-1)} \leq \Lambda \leq \rR(\Gamma) \leq \Gamma$ are characteristic and hence $A \leq \Gamma$ is characteristic.  In particular, $A$ is normal in $\Gamma$.

  Denote by $(\cG, \cE)$ the twisted groupoid constructed from $(\Gamma, A, \sigma)$ as in \cref{def:groupoid-decomposition}.  By \cref{prop:identification-algebras-groupoid-decomposition}, there is a commutative diagram
  \begin{center}
    \begin{tikzcd}
      \lone(\Gamma, \sigma) \arrow[r] \arrow[d, hookrightarrow] \arrow[rr, hookrightarrow, bend left] & \Cstarred(\Gamma, \sigma) \arrow[r, twoheadrightarrow] \arrow[d, "\cong"] & \cA \\
      \Lone(\cG, \cE) \arrow[r] & \Cstarred(\cG, \cE) \arrow[ur, twoheadrightarrow, "\pi"]
    \end{tikzcd}
  \end{center}
  We need to prove that $\pi$ is injective.

  Since finitely generated abelian groups are \Cstar-unique by \cref{thm:polynomial-growth-group-cstar-unique}, the restriction of $\pi$ to $\cont(\hat A)$ is injective.  Let $D = \mathrm{Tors}(\hat A)$ be the torsion subgroup of $\hat A$, which is dense, because $A$ is a free abelian group.  By \cref{theorem:fiber-groups-sufficient} it suffices to prove that for all $\chi \in D$ the induced map $\pi_\chi \colon \Cstarred(\IsoInt{\cG}_\chi, \IsoInt{\cE}_\chi) \to \pi(\Cstarred(\IsoInt{\cG}, \IsoInt{\cE}))_\chi$ is injective on $\lone(\IsoInt{\cG}_\chi, \IsoInt{\cE}_\chi)$.

  Fix $\chi \in D$ and consider the neighbourhood stabiliser $\Gamma^\circ_\chi = \{g \in \Gamma \mid \exists \chi \in U \text{ open} \colon g|_U = \id_U\}$ for the action $\Gamma \grpaction{} \hat A$.  Observe that $A \leq \Gamma^\circ_\chi$.  By \cref{prop:groupoid-decomposition-stabilisers}, the inclusion $\lone(\IsoInt{\cG}_\chi, \IsoInt{\cE}_\chi) \subseteq \Cstarred(\IsoInt{\cG}_\chi, \IsoInt{\cE}_\chi)$ is isomorphic with $\lone(\Gamma^\circ_\chi/A, (\chi \circ \rho) \cdot (\sigma \circ (s \times s))) \subseteq \Cstarred(\Gamma^\circ_\chi/A, (\chi \circ \rho) \cdot (\sigma \circ (s \times s)))$, where $s \colon \Gamma^\circ_\chi/A \to \Gamma^\circ_\chi$ is a section and $\rho \in \rZ^2(\Gamma^\circ_\chi/A, A)$ the associated extension cocycle. We write $\tilde \sigma = (\chi \circ \rho) \cdot (\sigma \circ (s \times s))$.

  Let $(\cH, \cF)$ be the twisted groupoid associated with $(\Gamma^\circ_\chi, A, \sigma)$.  We have an inclusion of twisted groupoids $(\IsoInt{\cG}, \IsoInt{\cE}) \hra (\cH, \cF) \hra (\cG, \cE)$.  Let $I \unlhd \pi(\Cstarred(\cH, \cF))$ be the ideal generated by $\conto(\hat A \setminus \{\chi\})$ and observe that we have a commutative diagram
  \begin{center}
    \begin{tikzcd}
      \pi(\Cstarred(\IsoInt{\cG}, \IsoInt{\cE})) \arrow[r] \arrow[d, hookrightarrow] & \pi(\Cstarred(\IsoInt{\cG}, \IsoInt{\cE}))_\chi  \arrow[d, "\cong"] \\
      \pi(\Cstarred(\cH, \cF)) \arrow[r] & \pi(\Cstarred(\cH, \cF))/I 
    \end{tikzcd}
  \end{center}
  We write $\cB = \pi(\Cstarred(\cH, \cF))$ and $\cB/I = \cB_\chi$.  By \cref{lem:L1-isometric-isomorphism}, the kernel of the restriction map $\res_\chi\colon \Lone(\cH, \cF) \to \lone(\cH_\chi, \cF_\chi) \cong \lone(\Gamma^\circ_\chi/A, \tilde \sigma)$ is the ideal $J$ generated by the subalgebra ${\conto(\unitspace{\cH} \setminus \{\chi\})} = \conto(\hat A \setminus \{\chi\})$.  Using the fact that we have a commutative diagram
  \begin{center}
    \begin{tikzcd}
      \lone(\Gamma^\circ_\chi, \sigma) \arrow[dr, twoheadrightarrow] \arrow[d, hookrightarrow]  \\
      \Lone(\cH, \cF) \arrow[r, "\res_\chi"] & \lone(\Gamma^\circ_\chi/A, \tilde \sigma)
    \end{tikzcd}
  \end{center}
  we infer that the injection $\lone(\Gamma^\circ_\chi, \sigma) \hra \cB \subseteq \cA$ when dividing by $J \cap \lone(\Gamma^\circ_\chi, \sigma)$ and $I = \pi(J)$ descends to an injection $\lone(\Gamma^\circ_\chi/A, \tilde  \sigma) \hra \cB_\chi$.  So the induction hypothesis can be applied, since $A$ being infinite, the Hirsch length of every subgroup of $\Gamma^\circ_\chi/A$ is at most $l - 1$. So we have shown that we have a commutative diagram
  \begin{center}
    \begin{tikzcd}
      \lone(\Gamma^\circ_\chi/A, \tilde \sigma) \arrow[r, hookrightarrow] \arrow[d, "\cong"] & \cB_\chi \arrow[d, "\cong"] \\
      \lone(\IsoInt{\cG}_\chi, \IsoInt{\cE}_\chi) \arrow[r, "\pi_\chi"] & \pi(\Cstarred(\IsoInt{\cG}, \IsoInt{\cE}))_\chi
    \end{tikzcd}
  \end{center}
  which implies what we had to show. 
\end{proof}
We now describe several classes of groups to which \cref{thm:ideal-intersection-general-form} applies.  Our first application concerns the large class of acylindrically hyperbolic groups.  We remark that the \lone-ideal intersection property for their group algebras can be deduced directly from \cref{lem:finite-by-cstar-simple}, while the general statement for dynamical systems could be deduced using solely \cref{lem:update-crossed-products} and \Cstar-uniqueness of virtually cyclic groups.
\begin{corollary}
  \label{cor:acylindrically hyperbolic-groups-lone-ideal-intersection-property}
  Let $\Gamma$ be an acylindrically hyperbolic group.  Then any action of $\Gamma$ on a locally compact Hausdorff space has the \lone-ideal intersection property.
\end{corollary}
\begin{proof}
  In order to apply \cref{thm:ideal-intersection-general-form}, we need to check all conditions of \cref{def:polycyclically-bounded-groups}.  By \cite[Theorem 2.35]{dahmaniguirardelosin17} combined with \cref{prop:cstar-simple-quotient} the first condition is satisfied.  The second and third conditions are satisfied thanks to \cite[Theorem 1.1]{osin16}, which says that subgroups of acylindrically hyperbolic groups are virtually cyclic or contain a copy of the free group.
\end{proof}
In order to obtain our next class of examples to which our main result applies, we need the following result, which is folklore.  We refer the reader unfamiliar with Lie theory to \cite[Table 9, p.~312-317]{onishchikvinberg1990} for the classification of simple real Lie algebras and their rank, which by definition is the dimension of a maximal $\RR$-diagonalisable Lie subalgebra.
\begin{proposition}
  \label{prop:solvable-subgroups-polycyclic}
  Let $\Gamma$ be a lattice in a connected Lie group.  Then there is $l \in \NN$ such that every solvable subgroup of $\Gamma$ is virtually polycyclic and has Hirsch length at most $l$.
\end{proposition}
\begin{proof}
  Let $G$ be a connected Lie group in which $\Gamma$ is a lattice.  By \cite[Lemma 6]{prasad1976-Lie-groups}, there is a normal subgroup $\Lambda \unlhd \Gamma$ such that $\Lambda$ is virtually a lattice in a connected solvable Lie group and $\Gamma/\Lambda$ is a lattice in a connected semisimple Lie group with trivial centre and without compact factors.  By \cite[Proposition 3.7]{raghunathan1972} every lattice in a connected simply connected solvable Lie group is polycyclic of Hirsch length bounded by the dimension of the Lie group.  Since every connected solvable Lie group is a quotient by a central discrete subgroup of its universal cover, the conclusion applies to lattices in arbitrary connected solvable Lie groups.  So we may assume for the rest of the proof that $\Gamma$ is a lattice in a connected semisimple Lie group $G$ with trivial centre and without compact factors.

  Passing to a finite index subgroup of $\Gamma$, there are direct product decompositions $G = \prod_{i = 1}^n G_i$ and $\Gamma = \prod_{i = 1}^n \Gamma_i$ such that $\Gamma_i \leq G_i$ is an irreducible lattice \cite[Theorem 5.22]{raghunathan1972}.  It hence suffices to consider the case where $\Gamma \leq G$ is already irreducible.  Assuming that $G$ is locally isomorphic with $\mathrm{SO}^+(n,1)$ or $\mathrm{SU}(n,1)$, the group $\Gamma$ acts on the hyperbolic boundary of $G$.  Thus, every solvable subgroup of $G$ is virtually cyclic, finishing the proof in this case. Assume that $G$ is not locally isomorphic with either $\mathrm{SO}^+(n,1)$ or $\mathrm{SU}(n,1)$.  Then the arithmeticity theorems of Margulis for lattices in semisimple Lie groups of higher rank presented in \cite[Chapter IX]{margulis1991} and \cite[Theorem 6.1.2]{zimmer84}, and the arithmeticity theorem for simple Lie groups of rank one locally isomorphic with $\mathrm{Sp}(n,1)$ or $\rF_{4(-20)}$ by Corlette \cite{corlette1992} and Gromov-Schoen \cite{gromovschoen1992} applies to show that $\Gamma$ is virtually linear over $\ZZ$.  Say it virtually embeds into  $\GL_n(\ZZ)$.  Now \cite[Proposition 2.9]{detinkoflanneryobrien2013} says that there is $l = l(n)$ such that every solvable subgroup of $\GL_n(\ZZ)$ is polycyclic of Hirsch length at most $l$.
\end{proof}

\begin{corollary}
  \label{cor:lattices-lone-ideal-intersection-property}
  Let $\Gamma$ be a lattice in a connected Lie group.  Then any action of $\Gamma$ on a locally compact Hausdorff space has the \lone-ideal intersection property.
\end{corollary}
\begin{proof}
  In order to apply \cref{thm:ideal-intersection-general-form}, we have to check all conditions of \cref{def:polycyclically-bounded-groups}.  Let $\Lambda$ be the amenable radical of $\Gamma$.  Then by \cite[Lemma 6]{prasad1976-Lie-groups}, we infer that $\Lambda$ is virtually a lattice in a connected solvable Lie group and that $\Gamma / \Lambda$ is a lattice in a semisimple Lie group with trivial centre and without compact factors.  Since Lie groups with trivial centre are linear, \cite[Theorem 6.9]{breuillardkalantarkennedyozawa14} implies that $\Gamma/\Lambda$ is \Cstar-simple.  So the first condition of \cref{def:polycyclically-bounded-groups} is verified thanks to \cref{prop:cstar-simple-quotient}.  Also, the Tits alternative for linear groups in characteristic zero \cite{tits1972-alternative} shows that every amenable subgroup of $\Gamma/\Lambda$ is virtually solvable.  Since $\Lambda$ is virtually solvable, this shows that every amenable subgroup of $\Gamma$ is virtually solvable.  This checks the second condition of \cref{def:polycyclically-bounded-groups}.  In order to verify the last one, we can apply \cref{prop:solvable-subgroups-polycyclic}.
\end{proof}
A variation of the core arguments in the previous theorem, also covers many linear groups.
\begin{corollary}
  \label{cor:linear-groups-lone-ideal-intersection-property}
  Let $\Gamma$ be a linear group over the integers of a number field. Then any action of $\Gamma$ on a locally compact Hausdorff space has the \lone-ideal intersection property.
\end{corollary}
\begin{proof}
  Let $\Gamma$ be as in the statement of the theorem.  We have to check all three conditions of \cref{def:polycyclically-bounded-groups}.  The first condition is satisfied thanks to \cref{prop:cstar-simple-quotient} combined with \cite[Theorem 6.9]{breuillardkalantarkennedyozawa14}.  The second condition holds thanks to the Tits alternative for linear groups in characteristic zero \cite{tits1972-alternative}.  The last condition holds thanks to \cite[Proposition 2.9]{detinkoflanneryobrien2013}.  
\end{proof}
Our final class of examples to which \cref{thm:ideal-intersection-general-form} applies are virtually polycyclic groups, and more generally locally virtually polycyclic groups, which are precisely those groups whose finitely generated subgroups are virtually polycyclic.  We also state the result in terms of \Cstar-uniqueness.
\begin{corollary}
  \label{cor:locally-virtually-polycyclic}
  Let $\Gamma$ be a locally virtually polycyclic group.  Then any action of $\Gamma$ on a locally compact Hausdorff space has the \lone-ideal intersection property.  In particular, every locally virtually polycyclic group is \Cstar-unique.
\end{corollary}
\begin{proof}
  By \cref{prop:lone-ideal-intersection-directed-union}, it suffices to show that every virtually polycyclic group satisfies the conditions of \cref{def:polycyclically-bounded-groups}.  The first condition is satisfied since virtually polycyclic groups are amenable.  The second condition holds, since every subgroup of a polycyclic group is polycyclic.  Finally, the Hirsch length is monotone for inclusions of groups, so that the last condition is also satisfied.  Now \cref{thm:ideal-intersection-general-form} applies.
\end{proof}

\begin{remark}
  \label{rem:possible-extensions}
  It would be interesting to understand whether all linear groups have the \lone-ideal intersection property.  We expect that a positive answer can be obtained.  However, the groupoid techniques employed in the present work will likely not be sufficient to prove such a result for two reasons.  First, there need not be any torsion points in the dual of an abelian group, so that an induction like in the proof of \cref{thm:ideal-intersection-general-form} cannot be performed.  Second, following the strategy of the present work, there is no clear induction variable available for solvable groups which are not polycyclic.  The derived length is not suitable.  Indeed, the induction step in the proof of \cref{thm:ideal-intersection-general-form} only divides out a (possibly proper) infinite subgroup of the last term in the derived series.

Concrete examples of solvable, non-polycyclic groups can nevertheless be covered by our present methods.  The arguments presented show that metabelian groups have the \lone-ideal intersection property, since each such group is an inductive limit of semi-direct products $A \rtimes \ZZ^{n_i}$ for some monotone sequence of natural numbers $(n_i)_i$.  We don't give any details of the argument.  Many metabelian groups are already known to have the \lone-ideal intersection property by \cite[p. 11, Korollar]{boidolleptinschurmannwahle1978}.
\end{remark}

\begin{remark}
	Another interesting direction of research would approach cocycle twists of group $C^*$-algebras already known to have the $\ell^1$-ideal intersection property. For groups which are finite-by-(C*-simple) this is the content of \cref{lem:finite-by-cstar-simple}, however for virtually polycyclic groups, our proof techniques only apply to twists by cocycles with a finite image. The same obstruction is inherited to the case of lattices in Lie groups.
\end{remark}



{\small
  \printbibliography

@incollection{alekseev2019-MFO-report,
  author =	 "Alekseev, Vadim",
  title =	 "(Non)-uniqueness of {C}$^*$-norms on group rings of
                  amenable groups",
  booktitle =	 "{C}$^*$-algebras. {Abstracts} from the workshop
                  held {August} 11--17, 2019",
  editor =	 "R{\o}rdam, Mikael and Shlyakhtenko, Dimitri L. and
                  Thom, Andreas and Vaes, Stefaan",
  series =	 "Oberwolfach {R}ep.",
  volume =	 "16",
  number =	 "3",
  year =	 "2019",
  doi =		 "10.4171/OWR/2019/37",
}

@article{alekseevkyed2019,
  author =	 "Alekseev, Vadim and Kyed, David",
  title =	 "Uniqueness questions for {C}$^*$-norms on group
                  rings",
  journal =	 "Pac. J. Math.",
  volume =	 "298",
  number =	 "2",
  pages =	 "257-266",
  year =	 "2019",
  doi =		 "10.2140/pjm.2019.298.257",
}

@article{armstrong2022-uniqueness,
  author =	 "Armstrong, Becky",
  title =	 "A uniqueness theorem for twisted groupoid
                  {{\(\mathrm{C^\ast}\)}}-algebras",
  journal =	 "J. Funct. Anal.",
  volume =	 "283",
  number =	 "6",
  pages =	 "33",
  note =	 "Id/No 109551",
  year =	 "2022",
  doi =		 "10.1016/j.jfa.2022.109551",
}

@article{austadortega2022,
  author =	 "Austad, Are and Ortega, Eduard",
  title =	 "{{\(C^*\)}}-uniqueness results for groupoids",
  journal =	 "Int. Math. Res. Not.",
  volume =	 "2022",
  number =	 "4",
  pages =	 "3057-3073",
  year =	 "2022",
  doi =		 "10.1093/imrn/rnaa225",
}

@article{barnes1983,
  author =	 "Barnes, Bruce A.",
  title =	 "The properties *-regularity and uniqueness of
                  {C}$^*$-norm in a general *-algebra",
  journal =	 "Trans. Am. Math. Soc.",
  volume =	 "279",
  pages =	 "841-859",
  year =	 "1983",
  doi =		 "10.2307/1999571",
}

@article{boidol1980,
  author =	 "Boidol, Joachim",
  title =	 "*-regularity of exponential {L}ie groups",
  journal =	 "Invent. Math.",
  volume =	 "56",
  pages =	 "231-238",
  year =	 "1980",
  doi =		 "10.1007/BF01390046",
}

@article{boidol1984,
  author =	 "Boidol, Joachim",
  title =	 "Group algebras with a unique {C}$^*$-norm",
  journal =	 "J. Funct. Anal.",
  volume =	 "55",
  pages =	 "220-232",
  year =	 "1984",
  doi =		 "10.1016/0022-1236(84)90011-9",
}

@article{boidolleptinschurmannwahle1978,
  author =	 "Boidol, Joachim. and Leptin, Horst. and Sch{\"u}rmann, J{\"u}rgen. and
                  Vahle, D.",
  title =	 "R{\"a}ume primitiver {Ideale} in {Gruppenalgebren}",
  journal =	 "Math. Ann.",
  volume =	 "236",
  pages =	 "1-13",
  year =	 "1978",
  doi =		 "10.1007/BF01420252",
}

@misc{borys2019-boundary,
  author =	 "Borys, Clemens",
  title =	 "{The Furstenberg Boundary of a Groupoid.}",
  howpublished = "Preprint",
  eprint =	 "arXiv:1904.10062",
  year =	 "2019",
}

@article{breuillardkalantarkennedyozawa14,
  author =	 "Breuillard, Emmanuel and Kalantar, Mehrdad and
                  Kennedy, Matthew and Ozawa, Narutaka",
  title =	 "{C$^*$-simplicity and the unique trace property for
                  discrete groups.}",
  journal =	 "Publ. Math. Inst. Hautes \'Etud. Sci.",
  year =	 "2017",
  volume =	 "126",
  number =	 "1",
  pages =	 "35-71",
  doi =		 "10.1007/s10240-017-0091-2",
}

@article{brownnagyreznikoffsimswilliams2016,
  author =	 "Brown, Jonathan H. and Nagy, Gabriel and Reznikoff,
                  Sarah and Sims, Aidan and Williams, Dana P.",
  title =	 "Cartan subalgebras in {{\(C^\ast\)}}-algebras of
                  {Hausdorff} {\'e}tale groupoids",
  journal =	 "Integral Equations Oper. Theory",
  volume =	 "85",
  number =	 "1",
  pages =	 "109-126",
  year =	 "2016",
  doi =		 "10.1007/s00020-016-2285-2",
}

@article{bryderkennedy2018-twisted,
  author =	 "Bryder, Rasmus Sylvester and Kennedy, Matthew",
  title =	 "Reduced twisted crossed products over {C$^*$}-simple
                  groups",
  journal =	 "Int. Math. Res. Not.",
  volume =	 "2018",
  year =	 "2018",
  number =	 "6",
  pages =	 "1638-1655",
  doi =		 "10.1093/imrn/rnw296",
}

@article{corlette1992,
  author =	 "Corlette, Kevin",
  title =	 "Archimedean superrigidity and hyperbolic geometry",
  journal =	 "Ann. Math. (2)",
  volume =	 "135",
  number =	 "1",
  pages =	 "165-182",
  year =	 "1992",
  doi =		 "10.2307/2946567",
}

@article{detinkoflanneryobrien2013,
  author =	 "Detinko, A. S. and Flannery, D. L. and O'Brien,
                  E. A.",
  title =	 "Algorithms for linear groups of finite rank",
  journal =	 "J. Algebra",
  volume =	 "393",
  pages =	 "187-196",
  year =	 "2013",
  doi =		 "10.1016/j.jalgebra.2013.06.006",
}

@article{dahmaniguirardelosin17,
  author =	 "Dahmani, Fran{\c c}ois and Guirardel, Vincent and
                  Osin, Dennis",
  title =	 "{Hyperbolically embedded subgroups and rotating
                  families in groups acting on hyperbolic spaces.}",
  journal =	 "Mem. Am. Math. Soc.",
  volume =	 "245",
  year =	 "2017",
  number =	 "1156",
  pages =	 "152 pages",
  doi =		 "10.1090/memo/1156",
}

@article{deeleyputnamstrung2015,
  author =	 "Deeley, Robin J. and Putnam, Ian F. and Strung,
                  Karen R. ",
  title =	 "{Constructing minimal homeomorphisms on point-like
                  spaces and a dynamical presentation of the Jiang-Su
                  algebra.}",
  journal =	 "J. Reine Angew. Math.",
  volume =	 "742",
  pages =	 "241-261",
  year =	 "2015",
  doi =		 "10.1515/crelle-2015-0091",
}

@incollection{echterhoffwilliams2008,
  author =	 "Echterhoff, Siegfried and Williams, Dana P.",
  title =	 "The {M}ackey machine for crossed products: inducing
                  primitive ideals",
  booktitle =	 "Group representations, ergodic theory, and
                  mathematical physics. A tribute to {G}eorge
                  {W}. {M}ackey. {AMS} special session honoring the
                  memory of {G}eorge {W}. {M}ackey, New Orleans, LA,
                  USA, January 7--8, 2007",
  isbn =	 "978-0-8218-4225-6",
  pages =	 "129-136",
  year =	 "2008",
  publisher =	 "Providence, RI: American Mathematical Society",
}

@article{grigorchukmusatrordam17,
  author =	 "Grigorchuk, Rotislav and Musat, Magdalena and
                  R{\o}rdam, Mikael",
  title =	 "{Just-infinite C$^*$-algebras.}",
  journal =	 "Comment. Math. Helv.",
  volume =	 "93",
  number =	 "1",
  pages =	 "157-201",
  year =	 "2018",
  doi =		 "10.4171/CMH/432",
}

@article{gromovschoen1992,
  author =	 "Gromov, Mikhail and Schoen, Richard",
  title =	 "Harmonic maps into singular spaces and p-adic
                  superrigidity for lattices in groups of rank one",
  journal =	 "Publ. Math., Inst. Hautes {\'E}tud. Sci.",
  volume =	 "76",
  pages =	 "165-246",
  year =	 "1992",
  doi =		 "10.1007/BF02699433",
}

@inproceedings{haagerup15,
  author =	 "Haagerup, Uffe",
  title =	 "{A new look at $\mathrm{C}^*$-simplicity and the
                  unique trace property of a group.}",
  booktitle =	 "Operator Algebras and Applications",
  series =	 "Abel Symposia",
  volume =	 "12",
  editor =	 "Carlsen, Toke M. and Larsen, Nadia S. and Neshveyev,
                  Sergey and Skau, Christian",
  publisher =	 "Cham: Springer",
  pages =	 "167-176",
  year =	 "2016",
  doi =		 "10.1007/978-3-319-39286-8_7"
}

@article{hirshbergwinterzacharias2015,
  author =	 "Hirshberg, Ilan and Winter, Wilhelm and Zacharias,
                  Joachim",
  title =	 "Rokhlin dimension and {C}$^*$-dynamics",
  journal =	 "Commun. Math. Phys.",
  volume =	 "335",
  number =	 "2",
  pages =	 "637-670",
  year =	 "2015",
  doi =		 "10.1007/s00220-014-2264-x",
}

@article{kalantarkennedy14-boudaries,
  author =	 "Kalantar, Mehrdad and Kennedy, Matthew",
  title =	 "{Boundaries of reduced $\mathrm{C}^*$-algebras of
                  discrete groups.}",
  journal =	 "J. Reine Angew. Math.",
  volume =	 "727",
  pages =	 "247-267",
  year =	 "2017",
  doi =		 "10.1515/crelle-2014-0111",
}

@misc{kawabe17,
  author =	 "Kawabe, Takuya",
  title =	 "{Uniformly recurrent subgroups and the ideal
                  structure of reduced crossed products.}",
  eprint =	 "arXiv:1701.03413",
  year =	 "2017",
  howpublished = "Preprint",
}

@article{kennedy15-cstarsimplicity,
  author =	 "Kennedy, Matthew",
  title =	 "{An intrinsic characterization of
                  $\mathrm{C}^*$-simplicity.}",
  journal =	 "Ann. Sci. {\'E}c. Norm. Sup{\'e}r.",
  volume =	 "53",
  number =	 "5",
  pages =	 "1105-1119",
  year =	 "2020",
  doi =		 "10.24033/asens.2441",
}

@misc{kennedykimliraumursu2021,
  author =	 "Kennedy, Matthew and Kim, Se-Jin and Li, Xin and
                  Raum, Sven and Ursu, Dan",
  title =	 "The ideal intersection property for essential
                  groupoid {C$^*$}-algebras",
  eprint =	 "arXiv:2107:03980",
  howpublished = "Preprint",
  year =	 "2021",
}

@article{kerrszabo2020,
  author =	 "Kerr, David and Szab{\'o}, G{\'a}bor",
  title =	 "Almost finiteness and the small boundary property",
  journal =	 "Commun. Math. Phys.",
  volume =	 "374",
  number =	 "1",
  pages =	 "1-31",
  year =	 "2020",
  doi =		 "10.1007/s00220-019-03519-z",
}

@article{kumjian86,
  author =	 "Kumjian, Alexander",
  title =	 "{On C$^*$-diagonals.}",
  journal =	 "Can. J. Math.",
  volume =	 "38",
  number =	 "4",
  year =	 "1986",
  pages =	 "969-1008",
  doi =		 "10.4153/CJM-1986-048-0",
}

@article{leungng2004,
  author =	 "Leung, Chi-Wai and Ng, Chi-Keung",
  title =	 "Some permanence properties of {C}$^*$-unique groups",
  journal =	 "J. Funct. Anal.",
  volume =	 "210",
  number =	 "2",
  pages =	 "376-390",
  year =	 "2004",
  doi =		 "10.1016/j.jfa.2003.11.003",
}

@book{margulis1991,
  author =	 "Margulis, Gregori Aleksandrovitch",
  title =	 "Discrete subgroups of semisimple {L}ie groups",
  series =	 "Ergebnisse der Mathematik und ihrer Grenzgebiete,
                  3. Folge",
  volume =	 "17",
  publisher =	 "Berlin etc.: Springer-Verlag",
  pages =	 "388 p.",
  year =	 "1991",
  isbn =	 "3-540-12179-X",
}

@article{nilsen1996,
  author =	 "Nilsen, May",
  title =	 "{{\(\mathrm{C}^*\)}}-bundles and {{\(\mathrm{C}_
                  0(X)\)}}-algebras",
  journal =	 "Indiana Univ. Math. J.",
  volume =	 "45",
  number =	 "2",
  pages =	 "463-477",
  year =	 "1996",
  doi =		 "10.1512/iumj.1996.45.1086",
}

@book{onishchikvinberg1990,
  author =	 {Onishchik, Arkadij L. and Vinberg, Ernest B.},
  title =	 "{L}ie groups and algebraic groups",
  note =	 "{T}ranslated from the {R}ussian by
                  {D}. {A}. {L}eites",
		  series = "Springer Series in Soviet Mathematics",
  isbn =	 "3-540-50614-4",
  year =	 "1990",
  publisher =	 "Berlin etc.: Springer-Verlag",
  doi =		 "10.1007/978-3-642-74334-4",
}

@article{osin16,
  author =	 "Osin, Denis",
  title =	 "Acylindrally hyperbolic groups",
  journal =	 "Trans. Am. Math. Soc.",
  volume =	 "368",
  number =	 "2",
  pages =	 "851-888",
  year =	 "2016",
  doi =		 "10.1090/tran/6343",
}

@article{packerraeburn89,
  author =	 "Packer, Judith A. and Raeburn, Iain",
  title =	 "{Twisted crossed products of
                  $\mathrm{C}^*$-algeras.}",
  journal =	 "Math. Proc. Camb. Philos. Soc.",
  volume =	 "106",
  number =	 "2",
  year =	 "1989",
  pages =	 "293-311",
  doi =		 {10.1017/S0305004100078129},
}

@article{prasad1976-Lie-groups,
  author =	 "Prasad, Gopal",
  title =	 "Discrete subgroups isomorphic to lattices in {L}ie
                  groups",
  journal =	 "Am. J. Math.",
  volume =	 "98",
  pages =	 "853-863",
  year =	 "1976",
  doi =		 "10.2307/2374033",
}

@book{raghunathan1972,
  author =	 "Raghunathan, M. S.",
  title =	 "Discrete subgroups of {L}ie groups",
  series =	 "Ergeb. Math. Grenzgeb.",
  volume =	 "68",
  year =	 "1972",
  publisher =	 "Berlin: Springer-Verlag",
  isbn =	 "978-3-642-86428-5",
}

@inproceedings{rosenberg94,
  author =	 "Rosenberg, Jonathan",
  title =	 "{$\mathrm{C}^*$-algebras and Mackey's theory of
                  group representations}",
  editor =	 "Robert S. Doran",
  booktitle =	 "$\mathrm{C}^*$-algebras: 1943-1993.  A fifty year
                  celebration.  AMS special session commemorating the
                  first fifty years of $\mathrm{C}^*$-algebra theory.
                  January 13-14, 1993. San Antonio, Texas.",
  series =	 " Contemporary Mathematics",
  volume =	 "167",
  year =	 "1994",
}

@article{scarparo2020-cstar-unique,
  author =	 "Scarparo, Eduardo",
  title =	 "A torsion-free algebraically {C}$^*$-unique group",
  journal =	 "Rocky Mt. J. Math.",
  volume =	 "50",
  number =	 "5",
  pages =	 "1813-1815",
  year =	 "2020",
  doi =		 "10.1216/rmj.2020.50.1813",
}

@book{segal1983,
  author =	 "Segal, Daniel",
  title =	 "Polycyclic groups",
  series =	 "Cambridge Tracts in Mathematics",
  volume =	 "82",
  year =	 "1983",
  publisher =	 "Cambridge etc.: Cambridge University Press",
  pages =	 "289 p.",
  isbn =	 "9780511565953",
  doi =		 "10.1017/CBO9780511565953",
}

@book{simsszabowilliams2020,
  author =	 "Sims, Aidan and Szab{\'o}, Gabor and Williams, Dana",
  title =	 "Operator algebras and dynamics: groupoids, crossed
                  products, and Rokhlin dimension.",
  editor =	 "Francesc Perera",
  series =	 "Advanced Courses in Mathematics CRM Barcelona",
  publisher =	 "Cham: Birkh{\"a}user/Springer",
  year =	 "2020",
  pages =	 "x+163 pp.",
  isbn =	 "978-3-030-39712-8",
  doi =		 "10.1007/978-3-030-39713-5"
}

@article{szabo2015-zn-actions,
  author =	 "Szab{\'o}, G{\'a}bor",
  title =	 "The {R}okhlin dimension of topological
                  $\mathbb{Z}^n$-actions",
  journal =	 "Proc. Lond. Math. Soc. (3)",
  volume =	 "110",
  number =	 "3",
  pages =	 "673-694",
  year =	 "2015",
  doi =		 "10.1112/plms/pdu065",
}

@article{tits1972-alternative,
  author =	 "Tits, Jacques",
  title =	 "Free subgroups in linear groups",
  journal =	 "J. Algebra",
  volume =	 "20",
  pages =	 "250-270",
  year =	 "1972",
  doi =		 "10.1016/0021-8693(72)90058-0",
}

@book{tomiyama1992-lecture-notes,
  author =	 "Tomiyama, Jun",
  title =	 "The interplay between topological dynamics and
                  theory of {C$^*$}-algebras",
  series =	 "Lecture Notes Series, Seoul. 2",
  publisher =	 "Seoul: Seoul National University, College of Natural
                  Sciences, Department of Mathematics",
  pages =	 "69 p.",
  year =	 "1992",
}

@book{zimmer84,
  author =	 "Zimmer, Robert J.",
  title =	 "{Ergodic theory and semisimple groups.}",
  series =	 "Monographs in Mathematics",
  volume =	 "81",
  publisher =	 "Boston-Basel-Stuttgart: Birkh{\"a}user",
  year =	 "1984",
}
}


\vspace{2em}
\begin{minipage}[t]{0.55\linewidth}
  \small
  Are Austad \\
  Department of Mathematics \\
  University of Oslo \\
  P.O. Box 1053 Blindern \\
  N-0316 Oslo \\
  Norway \\[1em]
  areaus@math.uio.no
\end{minipage}
\begin{minipage}[t]{0.45\linewidth}
  \small
  Sven Raum \\
  Institute of Mathematics \\
  University of Potsdam \\
  Karl-Liebknecht-Str. 24-25 \\
  D-14476 Potsdam \\
  Germany \\[1em]
  sven.raum@uni-potsdam.de
\end{minipage}

\end{document}